\documentclass{aims}
\usepackage{amsmath, amssymb}
\usepackage{paralist}
\usepackage{graphicx}
\usepackage[colorlinks=true]{hyperref}
\hypersetup{urlcolor=blue, citecolor=red}

\def\currentvolume{X}
 \def\currentissue{X}
  \def\currentyear{200X}
   \def\currentmonth{XX}
    \def\ppages{X--XX}

\newtheorem{theorem}{Theorem}[section]

\theoremstyle{definition}
\newtheorem{definition}[theorem]{Definition}

\newtheorem{example}{Example}

\newcommand{\subfiguretitle}[1]{\scriptsize{#1} \\ \vspace*{1mm}}
\newcommand{\mf}[1]{{\mathfrak{#1}}}
\newcommand{\pd}[2]{\frac{\partial#1}{\partial#2}}
\newcommand{\twovec}[2]{\begin{bmatrix}{#1}\\{#2}\end{bmatrix}}
\newcommand{\pre}[1]{\operatorname{\bullet}#1}
\newcommand{\post}[1]{#1\operatorname{\bullet}}
\newcommand{\indices}[1]{\left<#1\right>}
\providecommand{\coloneqq}{\mathrel{\mathop:}=}
\providecommand{\abs}[1]{\lvert#1\rvert}

\title[Signal-Flow Based Runge--Kutta Methods]
      {Signal-Flow Based Runge--Kutta Methods for the Simulation of Complex Networks}

\author[Stefan Klus]{}
\email{sklus@upb.de}


\begin{document}
\maketitle

\centerline{\scshape Stefan Klus }
\medskip
{\footnotesize
 \centerline{Institute for Industrial Mathematics,}
 \centerline{University of Paderborn,}
 \centerline{33095 Paderborn, Germany}
}

\bigskip

% The name of the associate editor will be entered by an editorial staff
% "Communicated by the associate editor name" is not needed for special issue.
\centerline{(Communicated by the associate editor name)}

\begin{abstract}
Complex dynamical networks appear in a wide range of physical, biological, and engineering systems. The coupling of subsystems with varying time scales often results in multirate behavior. During the simulation of highly integrated circuits, for example, only a few elements underlie changing signals whereas the major part---usually up to 80 or even 90 per cent---remains latent. Standard integration schemes discretize the entire circuit with a single step size which is mainly limited by the accuracy requirements of the rapidly changing subcircuits \cite{GFtM05}. It is of a particular interest to speed up the simulation without a significant loss of accuracy. By exploiting the latency of the system, only a fraction of the equations has to be formulated and solved at a given time point.

G\"unther and Rentrop \cite{GR94} suggest that multirate strategies must be based both on the numerical information of the integration scheme and on the topology of the circuit. In this paper, we will introduce a directed graph describing the interdependency of the underlying system and propose Runge--Kutta methods which utilize the signal flow of the system in order to identify and exploit inactive regions. Furthermore, we describe an extension of these methods to identify and exploit periodic subsystems.
\end{abstract}

\section{Introduction}

In this paper, we will consider initial value problems
\begin{equation} \label{eq:ODE}
    \begin{split}
        \dot{x}(t) &= f(t, x(t)), \\
        x(t_0) &= x_0,
    \end{split}
\end{equation}
with $ t \in \mathbb{I} \subseteq \mathbb{R} $ and $ f : \mathbb{I} \times \mathbb{D} \rightarrow \mathbb{R}^n $, $ \mathbb{D} \subseteq \mathbb{R}^n $. A fundamental class of numerical solvers are one-step methods of the form
\begin{equation}
    x^{m+1} = x^m + h \, \Phi(t^m, x^m, h),
\end{equation}
where $ \Phi $ is referred to as the \emph{increment function}. Important examples of one-step methods are Runge--Kutta methods. The increment function of a general $ s $-stage Runge--Kutta method is given by
\begin{subequations} \label{eq:RKODE}
\begin{equation}
    \Phi(t^m, x^m, h) = \sum_{q=1}^s b_q k_q,
\end{equation}
where
\begin{equation}
    k_q = f\big(t^m + c_q h, x^m + h \sum_{r=1}^s a_{qr} k_r\big).
\end{equation}
\end{subequations}
The coefficients $ a_{qr} $, $ b_q $, and $ c_q $ are often arranged in form of the so-called \emph{Butcher tableau}
\begin{equation}
    \begin{array}{c|c}
        c & A  \\ \hline \\[-1.8ex] % Workaround
          & b^T
    \end{array}
    \qquad
    \coloneqq
    \qquad
    \begin{array}{c|cccc}
        c_1    & a_{11} & a_{12} & \dots  & a_{1s} \\
        c_2    & a_{21} & a_{22} & \dots  & a_{2s} \\
        \vdots & \vdots & \vdots & \ddots & \vdots \\
        c_s    & a_{s1} & a_{s2} & \dots & a_{ss} \\ \hline
               & b_1    & b_2    & \dots & b_s
    \end{array}.
\end{equation}
If the matrix $ A $ is strictly lower triangular, then the Runge--Kutta method is called explicit. Otherwise, the method is said to be implicit.

\section{Time-driven ordinary differential equations}

Without loss of generality, the ordinary differential equation \eqref{eq:ODE} can be rewritten as
\begin{equation}
    \twovec{x_E}{\dot{x}_I} = \twovec{f_E(t)}{f_I(x_E, x_I)},
\end{equation}
with external variables $ x_E \in \mathbb{R}^{n_E} $ and internal variables $ x_I \in \mathbb{R}^{n_I} $. That is, we split the system into two subsystems and introduce additional variables which can be explicitly written as a function of the time $ t $. The dimension of the input vector $ x_E $ depends on the number of different time-dependent terms, the dimension of the internal vector $ x_I $ is equal to the number of equations of the original system. We introduce this partitioning to measure the influence of the input signals on the internal variables and to generate a model of the signal flow.

From now on, for the sake of simplicity, we will write the system---to which we will refer as a \emph{time-driven ordinary differential equation}---as
\begin{equation} \label{eq:TDODE}
    \twovec{x_E}{\dot{x}_I} = f(t, x), \text{ with } x = \twovec{x_E}{x_I} \text{ and } f = \twovec{f_E}{f_I}.
\end{equation}
Thus, $ x_{E, i} = x_i $ and $ x_{I, i} = x_{n_E + i} $. Let $ n = n_E + n_I $ denote the size of the whole system again.

For a time-driven ordinary differential equation, a one-step method is of the form
\begin{equation}
    \twovec{x_E^{m+1}}{x_I^{m+1}} = \twovec{x_E^m}{x_I^m} + \twovec{\Delta x_E^m}{\Delta x_I^m},
\end{equation}
with
\begin{equation} \label{eq:TDODE_update}
    \begin{split}
        \Delta x_E^m &= f_E(t^{m+1}) - f_E(t^m), \\
        \Delta x_I^m &= h \, \Phi(t^m, x^m, h).
    \end{split}
\end{equation}
The increment function of a Runge--Kutta method can now be rewritten as
\begin{subequations} \label{eq:RKTDODE}
\begin{equation}
    \Phi(t^m, x^m, h) = \sum_{q=1}^s b_q k_I^q,
\end{equation}
where
\begin{equation}
    \begin{split}
        k_E^q &= f_E(t^m + c_q h), \\
        k_I^q &= f_I\big(k_E^q, x_I^m + h \sum_{r=1}^s a_{qr} k_I^r\big).
    \end{split}
\end{equation}
\end{subequations}

\section{Dependency graph}
\label{sec:Dependency graph}

Given a time-driven ordinary differential equation, we want to analyze how changes of the input variables $ x_E $ affect the internal variables $ x_I $ and how the signals propagate through the system. To this end, we derive a directed graph which represents the structure of the system.

Define $ \indices{n} = \{ 1, \dots, n \} $ to be the set of indices. Since in general the functions $ f_i $, $ i \in \indices{n} $, do not depend on all variables $ x_j $, $ j \in \indices{n} $, we introduce input and output sets for each variable to describe the dependency on other variables.

\begin{definition}[Input and output sets]
Define the \emph{input set} of $ x_i $, $ i \in \indices{n} $, to be
\begin{equation}
    \pre{x_i} = \left\{ x_j \; \bigg| \; \pd{f_i}{x_j} \not\equiv 0, \, j \in \indices{n} \right \}.
\end{equation}
Analogously, define the \emph{output set} to be
\begin{equation}
    \post{x_i} = \left\{ x_j \; \bigg| \; \pd{f_j}{x_i} \not\equiv 0, \, j \in \indices{n} \right \}.
\end{equation}
\end{definition}

That is, the variable $ x_i $ depends on $ x_j $ if the value of $ x_j $ is required for the evaluation of $ f_i $. The input and output sets induce a directed graph with the vertices being the variables and the edges being the dependency relations between the variables.

\begin{definition}[Dependency graph]
For a given time-driven ordinary differential equation, define the \emph{dependency graph} by $ \mf{G}_d(f) = (\mf{V}_d, \mf{E}_d) $, with $ \mf{V}_d = \{ \mf{v}_1, \dots, \mf{v}_n \} $ and $ \mf{E}_d = \{ (\mf{v}_i, \mf{v}_j) \mid x_i \in \pre{x_j}, \; i, j \in \indices{n} \} $.
\end{definition}

If it is clear which differential equation is meant, we will simply write $ \mf{G}_d $. The dependency graph of large-scale dynamical networks can be very sparse since the subsystems are often strongly coupled inside but only connected to a few other subsystems of the network.

\begin{example} \hspace*{\fill}
\begin{enumerate}
\item Consider the linear differential equation
\begin{equation*}
    \ddddot{x}(t) = \dddot{x}(t) + \dot{x}(t),
\end{equation*}
which is equivalent to the first-order system
\begin{equation*}
    \begin{bmatrix}
        \dot{x}_1(t) \\
        \dot{x}_2(t) \\
        \dot{x}_3(t) \\
        \dot{x}_4(t)
    \end{bmatrix}
    =
    \underbrace{
    \begin{bmatrix}
        0 & 1 & 0 & 0 \\
        0 & 0 & 1 & 0 \\
        0 & 0 & 0 & 1 \\
        0 & 1 & 0 & 1
    \end{bmatrix}
    }_{\displaystyle A}
    \begin{bmatrix}
        x_1(t) \\
        x_2(t) \\
        x_3(t) \\
        x_4(t)
    \end{bmatrix}.
\end{equation*}
The input and output sets are
\begin{equation*}
    \begin{array}{ll}
        \pre{x_1} = \{ x_2 \},      & \post{x_1} = \varnothing, \\
        \pre{x_2} = \{ x_3 \},      & \post{x_2} = \{ x_1, x_4 \}, \\
        \pre{x_3} = \{ x_4 \},      & \post{x_3} = \{ x_2 \}, \\
        \pre{x_4} = \{ x_2, x_4 \}, & \post{x_4} = \{ x_3, x_4 \}.
    \end{array}
\end{equation*}

The differential equation is an equation of order three in $ \dot{x}(t) $. This can also be seen in the dependency graph, which is shown in Figure~\ref{fig:LinearSystemDependency}, since $ x_1 $ depends only on $ x_2 $ and can be obtained by integration. Moreover, the transposed system matrix $ A^T $ is the adjacency matrix of $ \mf{G}_d $, i.e.\ $ \mf{G}_d = \mf{G}(A^T) $.

\begin{figure}[htb]
    \centering
    \includegraphics[width=0.2\textwidth]{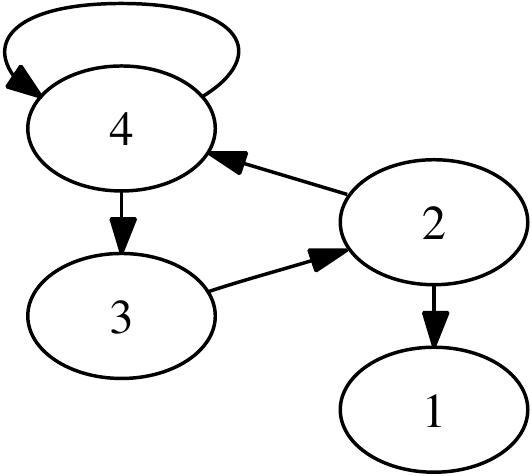}
    \caption{Dependency graph $ \mf{G}_d $ of the linear system.}
    \label{fig:LinearSystemDependency}
\end{figure}

\item Given the inverter chain of length $ N $ shown in Figure~\ref{fig:Inverterchain}, the corresponding circuit equations can be written as a time-driven ordinary differential equation with
\begin{equation*}
    f(t, v) =
    \left[
    \begin{array}{c}
        0 \\
        \mathrm{V_{dd}} \\
        V_s(t) \\ \hline
        g(v_1, v_2, v_3, v_4) \\
        g(v_1, v_2, v_4, v_5) \\
        \vdots \\
        g(v_1, v_2, v_{N+2}, v_{N+3})
    \end{array}
    \right].
\end{equation*}
Here, $ n_E = 3 $ and $ n_I = N $. The function $ g $ consists of the characteristic equations of the modules connected to the individual nodes and can be written as
\begin{equation*}
    g(v_1, v_2, v_{i-1}, v_i)
        = -\frac{1}{C_i}\big(\imath_{ds, n}(v_i, v_{i-1}, v_1) + \imath_{ds, p}(v_i, v_{i-1}, v_2)\big).
\end{equation*}
We use the Shichman--Hodges model \cite{SH68} to describe the drain-source current $ \imath_{ds} $ of the pMOS and nMOS transistors.

\begin{figure}[htb]
    \centering
    \includegraphics[width=0.7\textwidth]{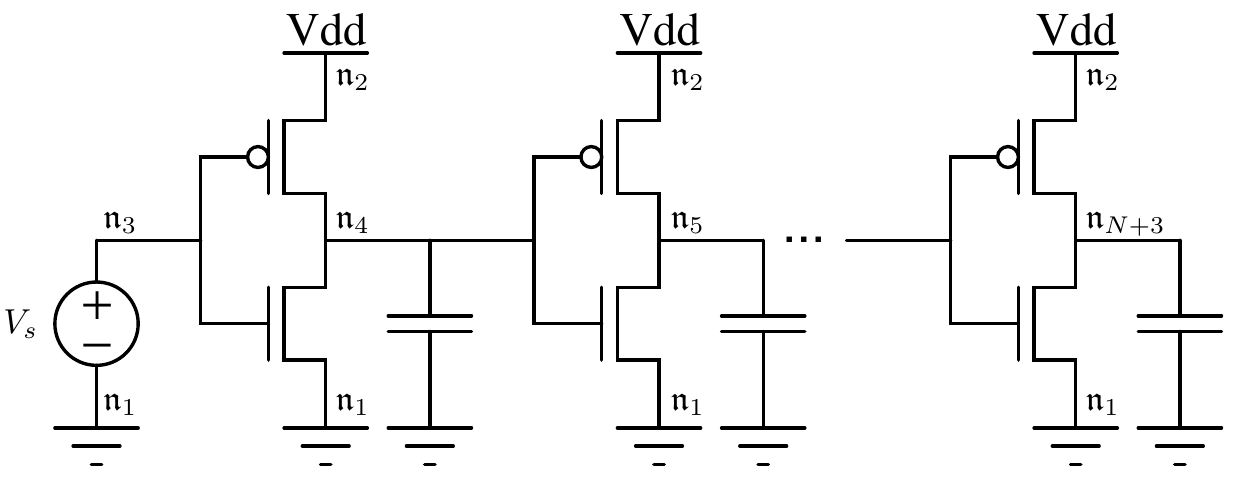}
    \caption{Inverter chain of length $ N $.}
    \label{fig:Inverterchain}
\end{figure}

Although the ground voltage and the positive supply voltage $ \mathrm{V_{dd}} $ are constant over time, we introduce additional variables since this assignment leads to a natural correlation between the nodes $ \mf{n}_i $ and the vertices $ \mf{v}_i $. In addition, it allows for a straightforward graph-based approach to generate the system of equations and the dependency graph. The Jacobian $ \pd{f}{v} $ exhibits the following structure
\begin{equation*}
    \pd{f}{v} =
    \left[
    \begin{array}{ccc|ccccc}
          &   &   &   &   &   &  &   \\
          &   &   &   &   &   &  &   \\
          &   &   &   &   &   &  &   \\ \hline
        * & * & * & * &   &   &  &   \\
        * & * &   & * & * &   &  &   \\
        * & * &   &   & * & * &  &   \\
        \vdots & \vdots &   &   &   & \ddots & \ddots & \\
        * & * &   &   &   &   & * & *
    \end{array}
    \right],
\end{equation*}
where empty places denote partial derivatives identical to zero. Figure~\ref{fig:InverterchainDependency} shows the dependency graph of the inverter chain. Since the constant voltages $ v_1 $ and $ v_2 $ have no influence on the dynamic signal flow, the corresponding vertices and associated edges have been omitted due to visualization reasons.

\begin{figure}[htb]
    \centering
    \includegraphics[width=0.5\textwidth]{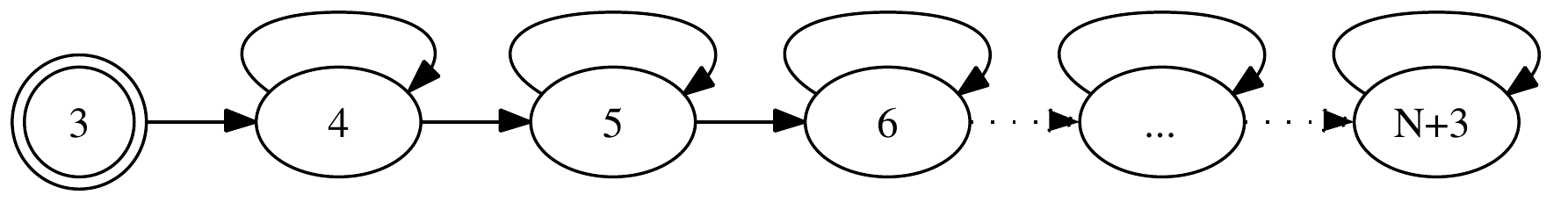}
    \caption{Dependency graph $ \mf{G}_d $ of the inverter chain.}
    \label{fig:InverterchainDependency}
\end{figure}
\end{enumerate}
\end{example}

In the following, we often identify $ x_i $ with $ \mf{v}_i $. Each internal vertex of the dependency graph represents a one-dimensional ordinary differential equation that is coupled to other one-dimensional systems. Generally speaking, a time-driven ordinary differential equation together with its dependency graph can be regarded as a coupled cell system~\cite{GPS04, GS03} with additional time-dependent inputs.

\section{Signal-flow based Runge--Kutta methods}

During the simulation of big and loosely coupled networks, different subsystems often exhibit different rates of activity. That is, the values in some parts of the network change rapidly, while in other parts the values change very slowly or do not change at all. The active regions usually vary over time so that a previously inactive region undergoes quick changes and vice versa.

Consider for example the inverter chain. If we apply an input signal, then, generally speaking, this input signal is reversed repeatedly with a small time delay so that it seems to flow continuously through the circuit. The step size control of standard integration schemes depends mainly on the fastest changing variables. As a result, even the inactive signals have to be recomputed at every time step unless multirate integration schemes or other techniques to exploit the latency are used. We will propose an integration scheme which utilizes the underlying structure of the system.

With the definitions in Section~\ref{sec:Dependency graph}, it is possible to determine which values of $ x^m $ are necessary to compute the new values of $ x^{m+1} $, namely, for the update of $ x_i^m $, all values of the variables of the input set $ \pre{x_i} $ are required. Since the external variables $ x_{E, i} $, $ i \in \indices{n_E} $, depend only on the time $ t $, the input sets are empty, i.e.\ $ \pre{x_{E, i}} = \varnothing $. The update of the internal values $ x_{I, i} $, $ i \in \indices{n_I} $, requires the evaluation of $ f_{I, i} $ and thus the values of $ \pre{x_{I, i}} $. To identify latent regions, we have to distinguish between the different vertex types.

\begin{definition}[Semi-latency]
Let $ t^m $ be the current time point and $ t^{m-1} $ the previous time point.
\begin{enumerate}
\item An external variable $ x_{E, i} $, $ i \in \indices{n_E} $, is said to be \emph{semi-latent} at $ t^m $ if
\begin{equation}
    f_{E, i}(t^m + c_q h) = f_{E, i}(t^{m-1} + c_q h)
\end{equation}
for all $ q = 1, \dots, s $.
\item An internal variable $ x_{I, i} $, $ i \in \indices{n_I} $, is defined to be \emph{semi-latent} if
\begin{equation}
    \Phi_i(t^{m-1}, x^{m-1}, h) = 0.
\end{equation}
\end{enumerate}
\end{definition}

The definition implies that $ x_{I, i}^m = x_{I, i}^{m-1} $ for all semi-latent internal variables. Whether a vertex is semi-latent at a specific time point is not known until all the values have been evaluated, but since our aim is to reduce the number of function evaluations, we want to mark vertices which need not be recomputed. Therefore, we introduce an additional concept.

\begin{definition}[Latency]
A variable $ x_i $, $ i \in \indices{n} $, is called \emph{latent of order} $ 1 $ if $ x_i $ and all variables of the set $ \pre{x_i} $ are semi-latent. Additionally, a latent variable $ x_i $ is defined to be \emph{latent of order} $ \nu $ if all variables in $ \pre{x_i} $ are at least latent of order $ \nu-1 $.
\end{definition}

Let $ \varepsilon $ be a user-defined error tolerance. For numerical computations, the semi-latency conditions are replaced by $ \abs{\Delta x_{E, i}^{m-1}} < \varepsilon $ and $ \abs{\Delta x_{I, i}^{m-1}} < \varepsilon $, respectively. In order to illustrate the different states of activity, we simulate the inverter chain.

\begin{example} \label{ex:InverterchainSimulation}
If the inverter chain is excited with a given input signal, then this signal flows---reversed at each inverter---through the circuit, as described above. Figure~\ref{fig:InverterchainSimulation} shows the voltages and activity states resulting when the circuit is excited with the displayed piecewise linear function. With a view to a better visualization, the respective activity states of the vertices are slightly shifted upward. Clearly, only a few vertices are active at each time point and these active regions flow through the dependency graph.

\begin{figure}[htbp]
    \begin{center}
        \subfiguretitle{a)}
        \includegraphics[width=0.9\textwidth]{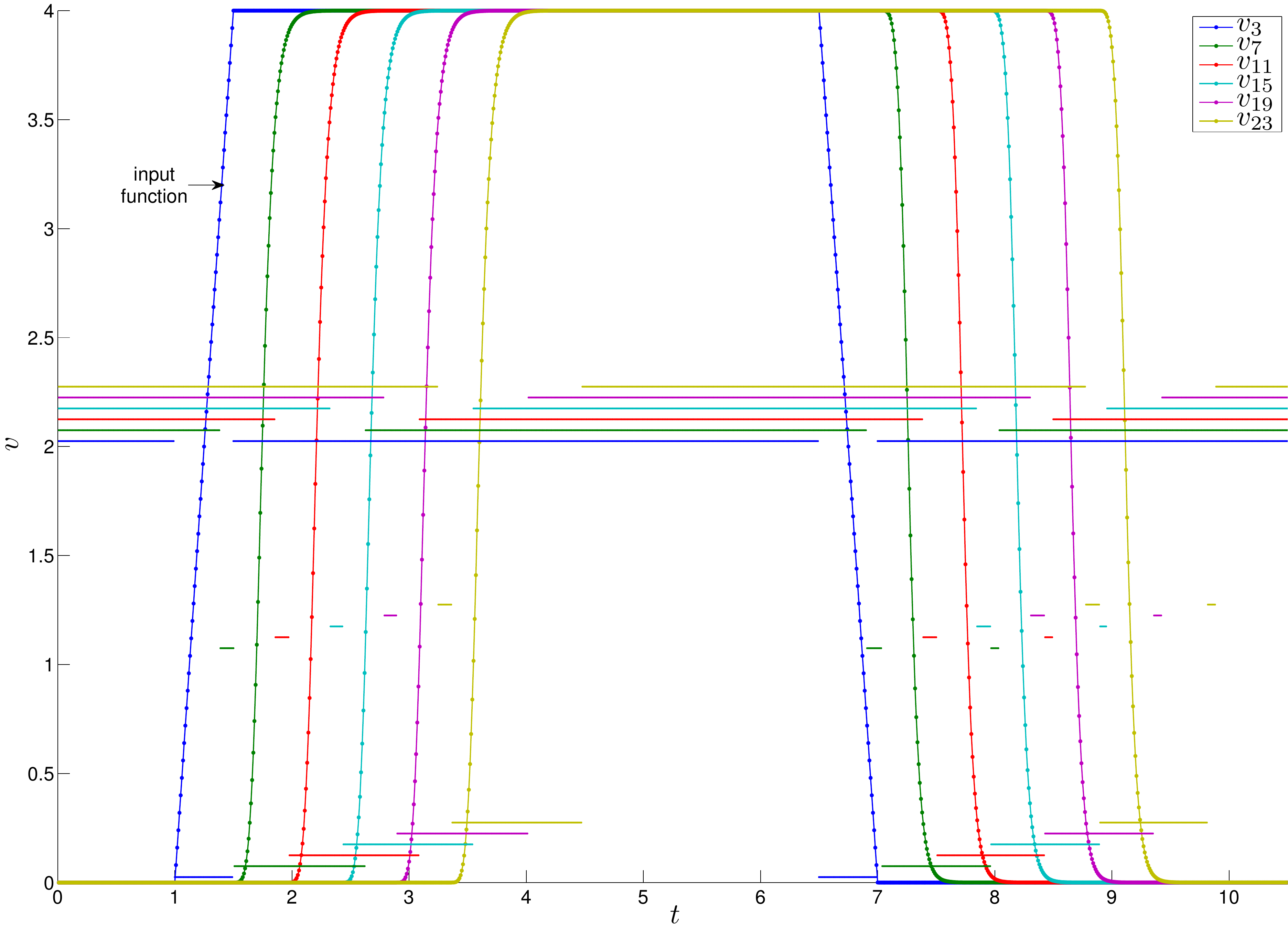} \\ \vspace*{4mm}
        
        \subfiguretitle{b)}
        \includegraphics[width=0.9\textwidth]{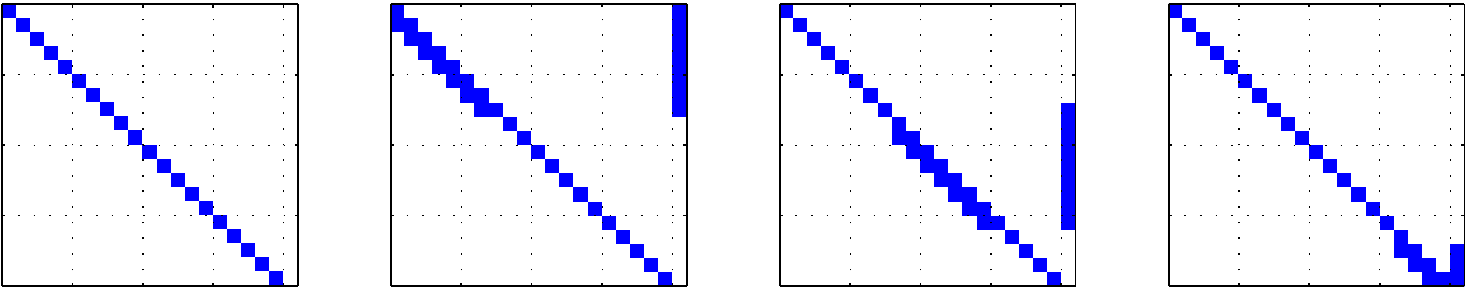} \\ \vspace*{4mm}
        
        \subfiguretitle{c)}
        \includegraphics[width=0.9\textwidth]{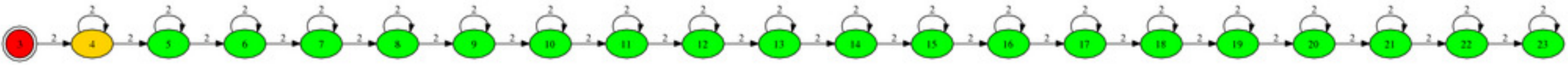} \\ \vspace*{1mm}
        \includegraphics[width=0.9\textwidth]{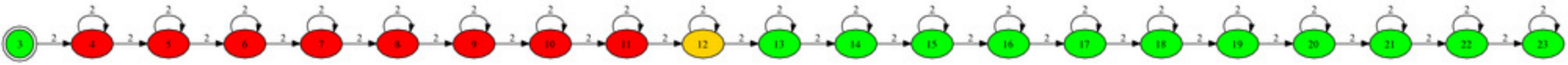} \\ \vspace*{1mm}
        \includegraphics[width=0.9\textwidth]{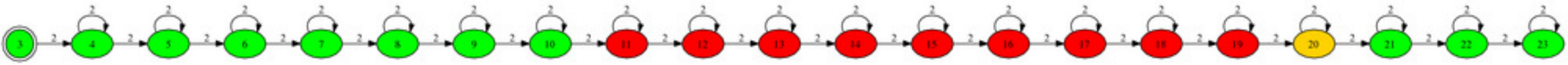} \\ \vspace*{1mm}
        \includegraphics[width=0.9\textwidth]{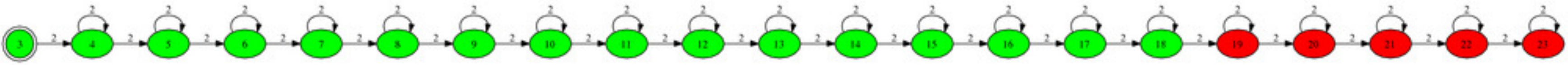} \\
        \caption[Excitation of the inverter chain with a piecewise linear function.]
                {Excitation of the inverter chain with a piecewise linear function. a) The dotted
                 trajectories show the input function and the voltages at intermediate vertices, the thin
                 horizontal lines in the corresponding color the activity state. Here, $ 0 $ denotes active,
                 $ 1 $ semi-latent, and $ 2 $ latent, respectively. b) Structure of $ \pd{f_I}{x_I} $ and
                 $ \dot{x}_I $ at time 1, 2, 3, and 4 for a threshold of $ 10^{-4} $. c) Activity states
                 at time 1, 2, 3, and 4, where red vertices represent active, yellow vertices semi-latent,
                 and green vertices latent regions.}
        \label{fig:InverterchainSimulation}
    \end{center}
\end{figure}
\end{example}

The example shows that the vertices are latent during the major part of the simulation, but each vertex at a different time. Below, we will propose modified Runge--Kutta methods for time-driven ordinary differential equations which take into account the dependency graph and the signal flow of the underlying system. The aim is to reduce the number of function evaluations without a huge loss of accuracy by exploiting the inherent latency. Since for some applications the function evaluations are time-consuming, whereas the update of the dependency graph can be accomplished in linear time, this approach offers the possibility to conceivably speed up the simulation.

\subsection{Explicit Runge--Kutta methods}

For the computation of the vectors $ k_E^q $ and $ k_I^q $, $ q = 1, \dots, s $, in \eqref{eq:RKTDODE}, it is necessary to evaluate the functions $ f_E $ and $ f_I $, respectively. The functions $ f_{I,i} $, $ i \in \indices{n_I} $, have to be recomputed if only one of the variables of the input set $ \pre{x}_{I, i} $ is active or semi-latent. If $ x_{I, i} $ is latent of a certain order, then we can reuse the previous value.

\begin{definition}[Signal-flow based Runge--Kutta method]
Given a time-driven ordinary differential equation, a \emph{signal-flow based Runge--Kutta method} is defined by
\begin{equation}
    \begin{split}
        x_E^{m+1} &= x_E^m + \Delta x_E^m, \\
        x_{I, i}^{m+1} &=
        \begin{cases}
            x_{I, i}^m,                     & \text{if } x_{I, i} \text{ is latent of order } s, \\
            x_{I, i}^m + \Delta x_{I, i}^m, & \text{otherwise},
        \end{cases}
    \end{split}
\end{equation}
for all $ i \in \indices{n_I} $. Here, $ s $ is again the number of stages. The vectors $ \Delta x_E^m $ and $ \Delta x_I^m $ are as defined in \eqref{eq:TDODE_update}.
\end{definition}

Provided that we use exact computation, the following theorem holds.

\begin{theorem}\label{th:ERK=sfERK}
The explicit Runge--Kutta methods and the corresponding signal-flow based methods are equivalent.
\end{theorem}
\begin{proof}
\allowdisplaybreaks
In the proof, we add the superscript $ m $ or $ m-1 $ to the stages to differentiate between the different time points. Let $ x_{I, i} $ be latent at $ t^m $, i.e.\ $ \Phi_i(t^{m-1}, x^{m-1}, h) = 0 $ and
\begin{align*}
    f_{E, j}(t^m + c_q h) = f_{E, j}(t^{m-1} + c_q h)
        & \; \Rightarrow \; k_{E, j}^{m, q} = k_{E, j}^{m-1, q} \quad \forall x_{E, j} \in \pre{x_{I, i}}, \\
    \Phi_j(t^{m-1}, x^{m-1}, h) = 0
        & \; \Rightarrow \; x_{I, j}^m = x_{I, j}^{m-1} \quad \forall x_{I, j} \in \pre{x_{I, i}}.
\end{align*}
For $ q = 1 $, we have $ c_1 = 0 $ and thus
\begin{equation*}
    k_{I, i}^{m, 1} = f_{I, i}(x_E^m, x_I^m) = f_{I, i}(x_E^{m-1}, x_I^{m-1}) = k_{I, i}^{m-1, 1}
\end{equation*}
since $ f_{I, i} $ depends only on the values of the input set $ \pre{x_{I, i}} $ and these values are the same as in the previous time step by definition. Now, assume that $ x_{I, i} $ is latent of order $ 2 $, i.e.\ all inputs of $ x_{I, i} $ are at least latent of order $ 1 $. If follows that
\begin{equation*}
    \begin{split}
        k_{I, i}^{m, 2} &= f_{I, i}(k_E^{m, 2}, x_I^m + h \, a_{21} k_I^{m, 1}) \\
                        &= f_{I, i}(k_E^{m-1, 2}, x_I^{m-1} + h \, a_{21} k_I^{m-1, 1}) = k_{I, i}^{m-1, 2}
    \end{split}
\end{equation*}
using the same reasoning again. Furthermore, by induction it can be shown that
\begin{equation*}
    \begin{split}
        k_{I, i}^{m, q} &= f_{I, i}\big(k_E^{m, q}, x_I^m + h \sum_{r=1}^{q-1} a_{qr} k_I^{m, r}\big) \\
                        &= f_{I, i}\big(k_E^{m-1, q}, x_I^{m-1} + h \sum_{r=1}^{q-1} a_{qr} k_I^{m-1, r}\big)
                         = k_{I, i}^{m-1, q}
    \end{split}
\end{equation*}
if $ x_{I, i} $ is latent of order $ q $ and
\begin{equation*}
    \begin{split}
        x_{I, i}^{m+1} &= x_{I, i}^m + h \, \Phi_i(t^m, x^m, h) \\
                       &= x_{I, i}^m + h \sum_{q=1}^s b_q k_{I, i}^{m, q} \\
                       &= x_{I, i}^{m-1} + h \sum_{q=1}^s b_q k_{I, i}^{m-1, q} \\
                       &= x_{I, i}^{m-1} + h \, \Phi_i(t^{m-1}, x^{m-1}, h) = x_{I, i}^m
    \end{split}
\end{equation*}
if $ x_{I, i} $ is latent of order $ s $.
\end{proof}

For numerical computations, we do not update a variable if it is latent of order at least one assuming that the influence of longer paths is negligibly small. In the following, we will abbreviate the standard classical fourth-order Runge--Kutta method as RK and the corresponding signal-flow based method as sfRK.

\begin{example} \label{ex:Inverterchain_sfRK}
Consider once again the inverter chain, which is a popular benchmark problem for multirate integration schemes. To analyze the efficiency of the signal-flow based standard Runge--Kutta method, we simulate the inverter chain of length $ N = 100 $ with variably time-consuming function evaluations and different rates of inherent latency. To vary the amount of latency, we apply periodic input functions with different delays between two adjacent pulse signals, as shown in Figure~\ref{fig:InverterchainPwlInput}. The complexity of the transistor model is increased by artificially adding terms which do not affect the solution of the system.

\begin{figure}[htb]
    \centering
    \includegraphics[width=0.5\textwidth]{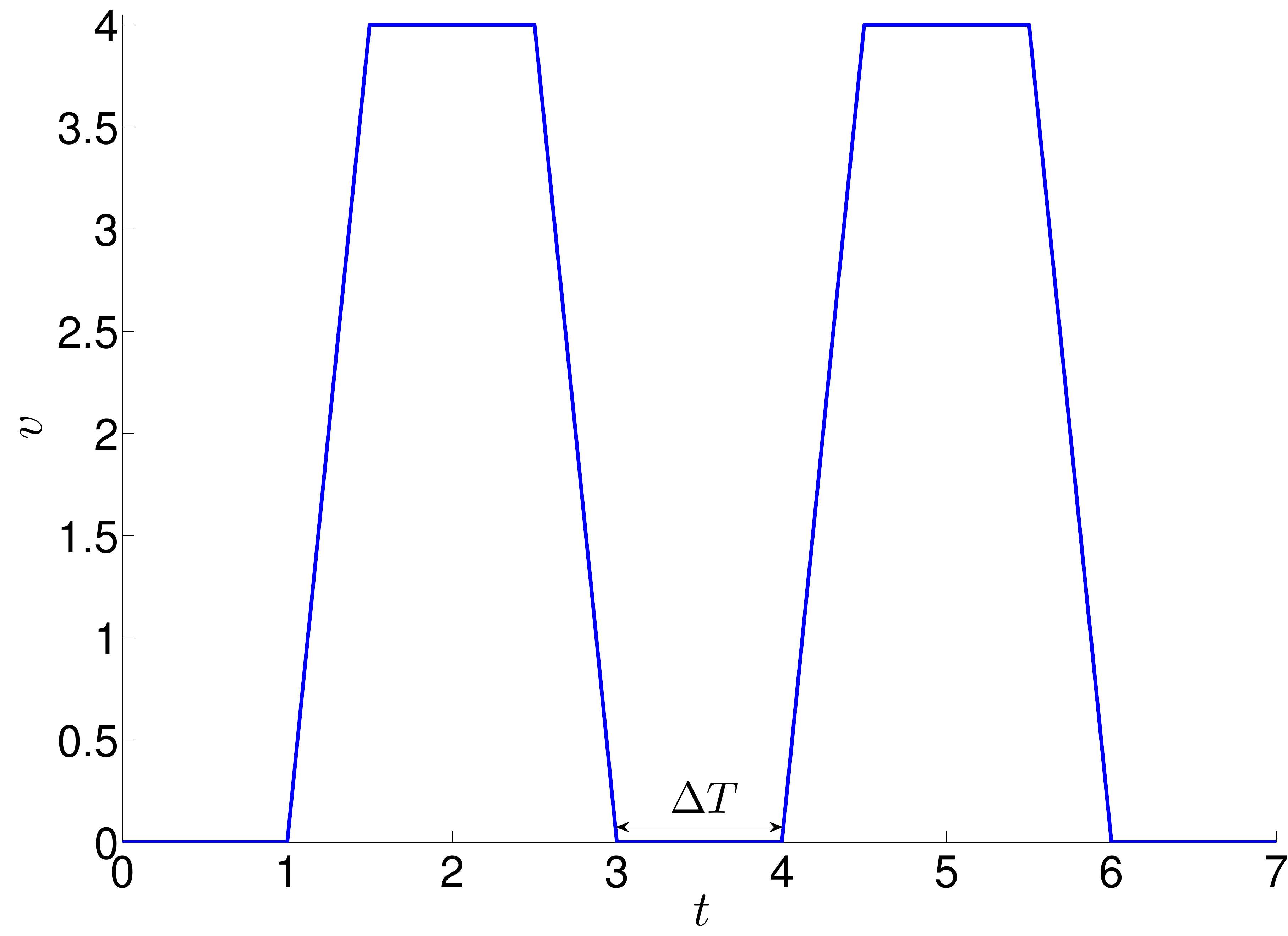}
    \caption{Piecewise linear input function with varying delay $ \Delta T $ to emulate latency.}
    \label{fig:InverterchainPwlInput}
\end{figure}

The runtimes of the simulation with both the standard Runge--Kutta method and the corresponding signal-flow based method for varying model complexities and input functions are shown in Figure~\ref{fig:Inverterchain_sfRK}. Here, the time interval is $ \mathbb{I} = [0, 40] $, the step size $ h = \frac{1}{100} $, and the latency parameter $ \varepsilon = 10^{-6} $. While the runtime of RK does not depend on the inherent latency, the runtime of sfRK decreases with increasing latency. Furthermore, the more complex the transistor model, the bigger the speedup of the signal-flow based integration scheme due to the reduced number of function evaluations. Table~\ref{tab:Inverterchain_sfRK} contains the number of transistor model evaluations for different values of $ \Delta T $. The influence of $ \varepsilon $ on the speedup of sfRK and the average difference per step between RK and sfRK for a fixed delay $ \Delta T = 10 $ are shown in Figure~\ref{fig:InverterchainEpsilon_RK}.

\begin{figure}[htbp]
    \begin{center}
        \begin{minipage}[c]{0.45\textwidth}
            \centering
            \subfiguretitle{RK} \vspace*{0.35em}
            \includegraphics[width=\textwidth]{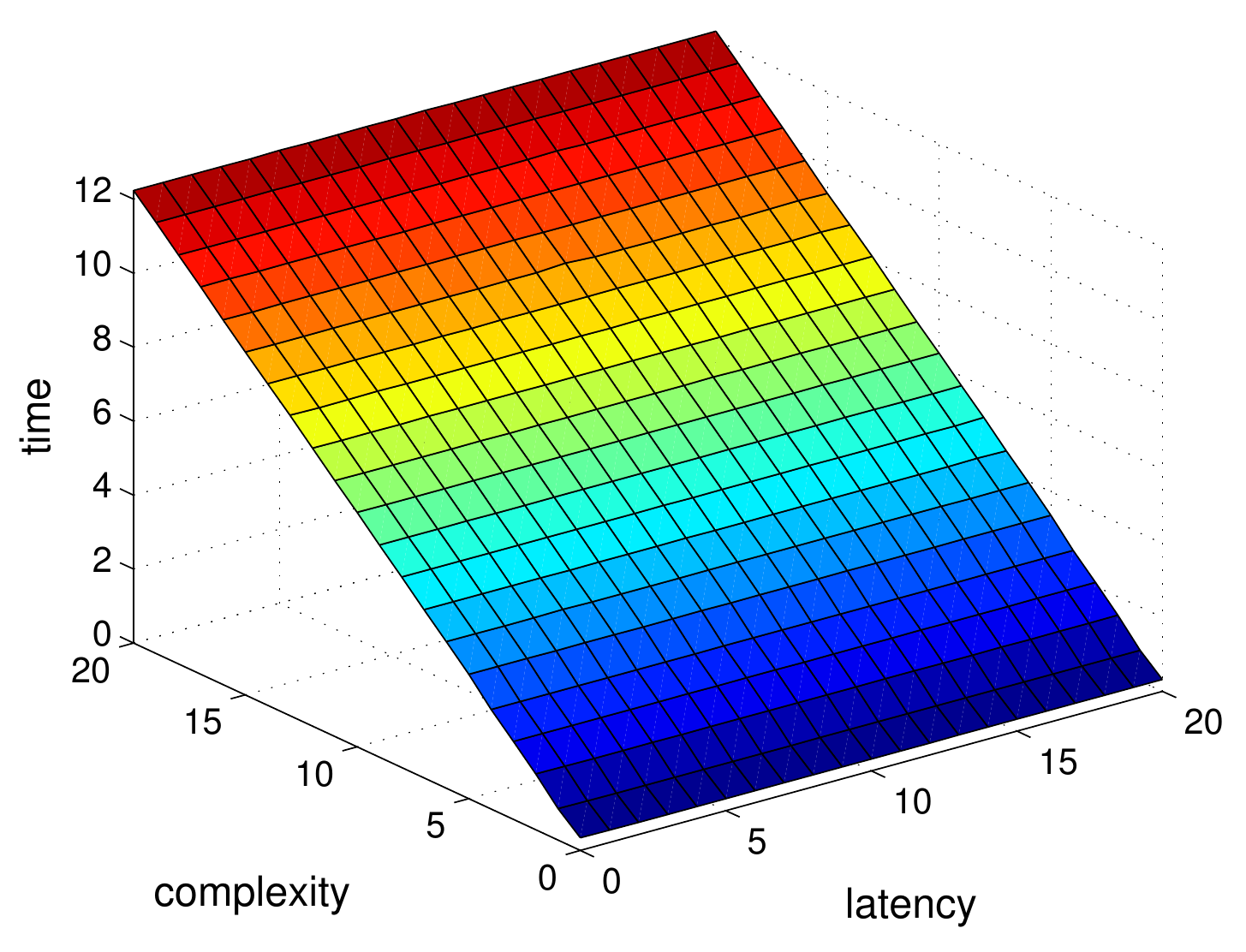}
        \end{minipage}
        \begin{minipage}[c]{0.45\textwidth}
            \centering
            \subfiguretitle{sfRK} \vspace*{0.35em}
            \includegraphics[width=\textwidth]{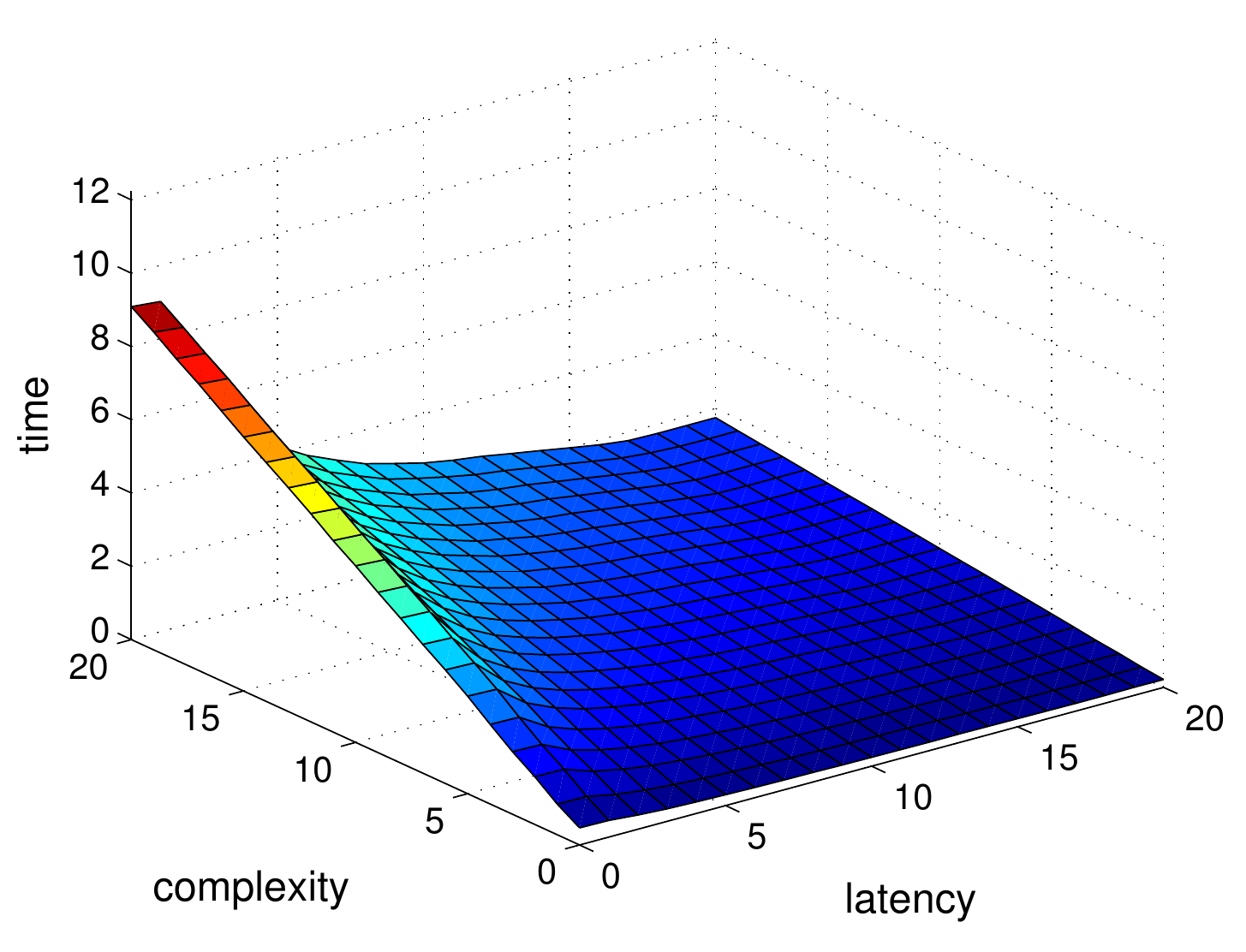}
        \end{minipage}
        \\ \vspace*{3mm}
        \begin{minipage}[c]{0.45\textwidth}
            \centering
            \subfiguretitle{RK vs. sfRK} \vspace*{0.35em}
            \includegraphics[width=\textwidth]{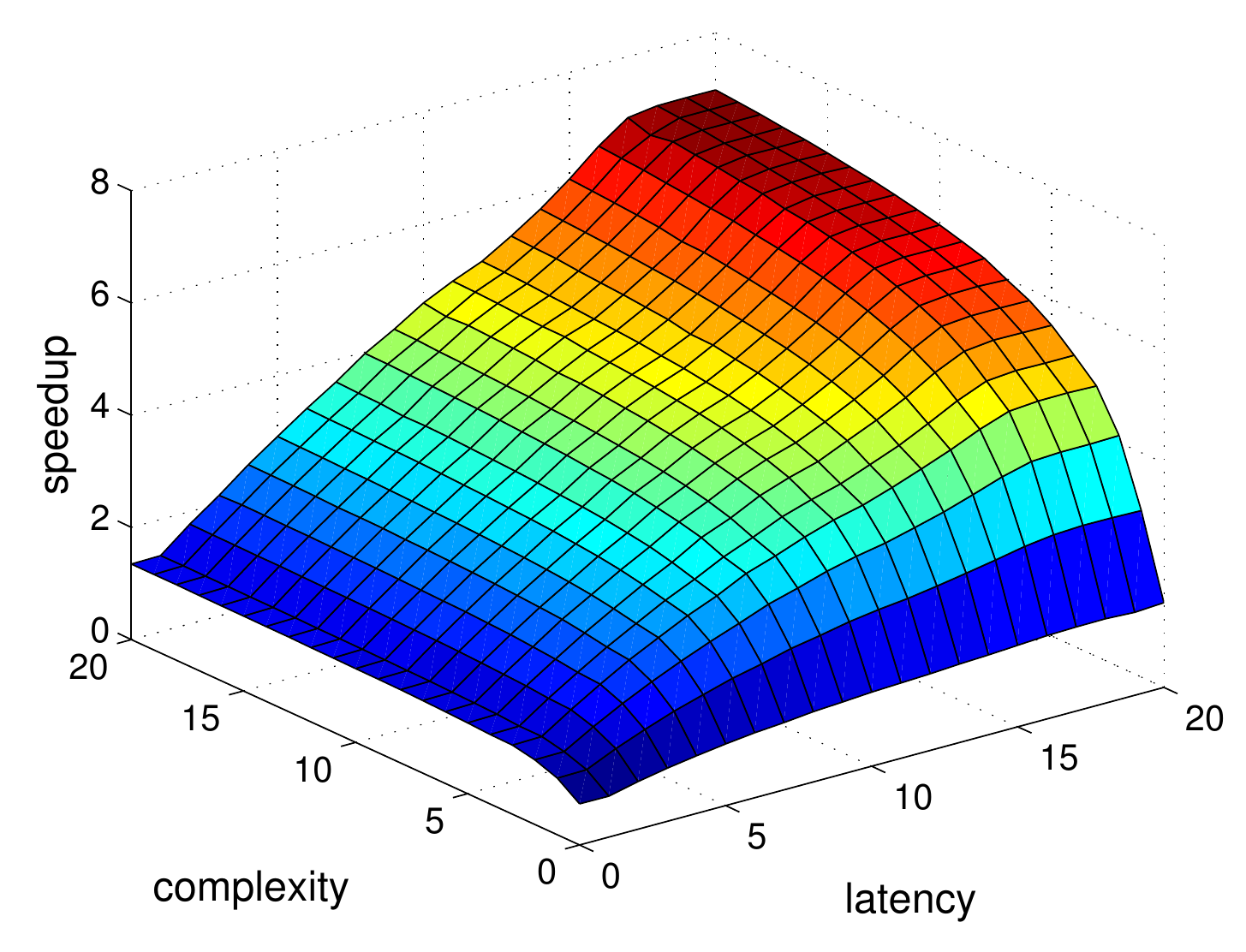}
        \end{minipage}
    \end{center}
    \caption{Influence of the complexity and latency on the runtime of RK and sfRK.}
    \label{fig:Inverterchain_sfRK}
\end{figure}

\begin{table}[htb]
    \caption{Number of transistor model evaluations of RK and sfRK.}
    \newcommand{\mc}[1]{\multicolumn{1}{|c|}{#1}}
    \newcommand{\ts}{\mspace{2mu}}
    \footnotesize
    \centering
    \begin{tabular}{|r*{5}{|r}|}
        \hline
        \mc{$ \Delta T $} & \mc{0} & \mc{5} & \mc{10} & \mc{15} & \mc{20} \\
        \hline
        \hline
                  RK & $ 3\ts200\ts000 $ & $ 3\ts200\ts000 $ & $ 3\ts200\ts000 $ & $ 3\ts200\ts000 $ & $ 3\ts200\ts000 $ \\
                sfRK & $ 2\ts317\ts152 $ & $ 1\ts046\ts664 $ & $     649\ts976 $ & $     479\ts360 $ & $     413\ts024 $ \\
        \hline
    \end{tabular}
    \label{tab:Inverterchain_sfRK}
\end{table}

We can reduce the number of function evaluations even for $ \Delta T = 0 $ since at the beginning of the simulation the circuit is in a steady state and it takes a short time until the input signal reaches the last inverter. During that time, parts of the circuit are inactive and need not be evaluated.

\begin{figure}[htb]
    \begin{center}
        \begin{minipage}[c]{0.45\textwidth}
            \centering
            \subfiguretitle{Speedup} \vspace*{0.35em}
            \includegraphics[width=\textwidth]{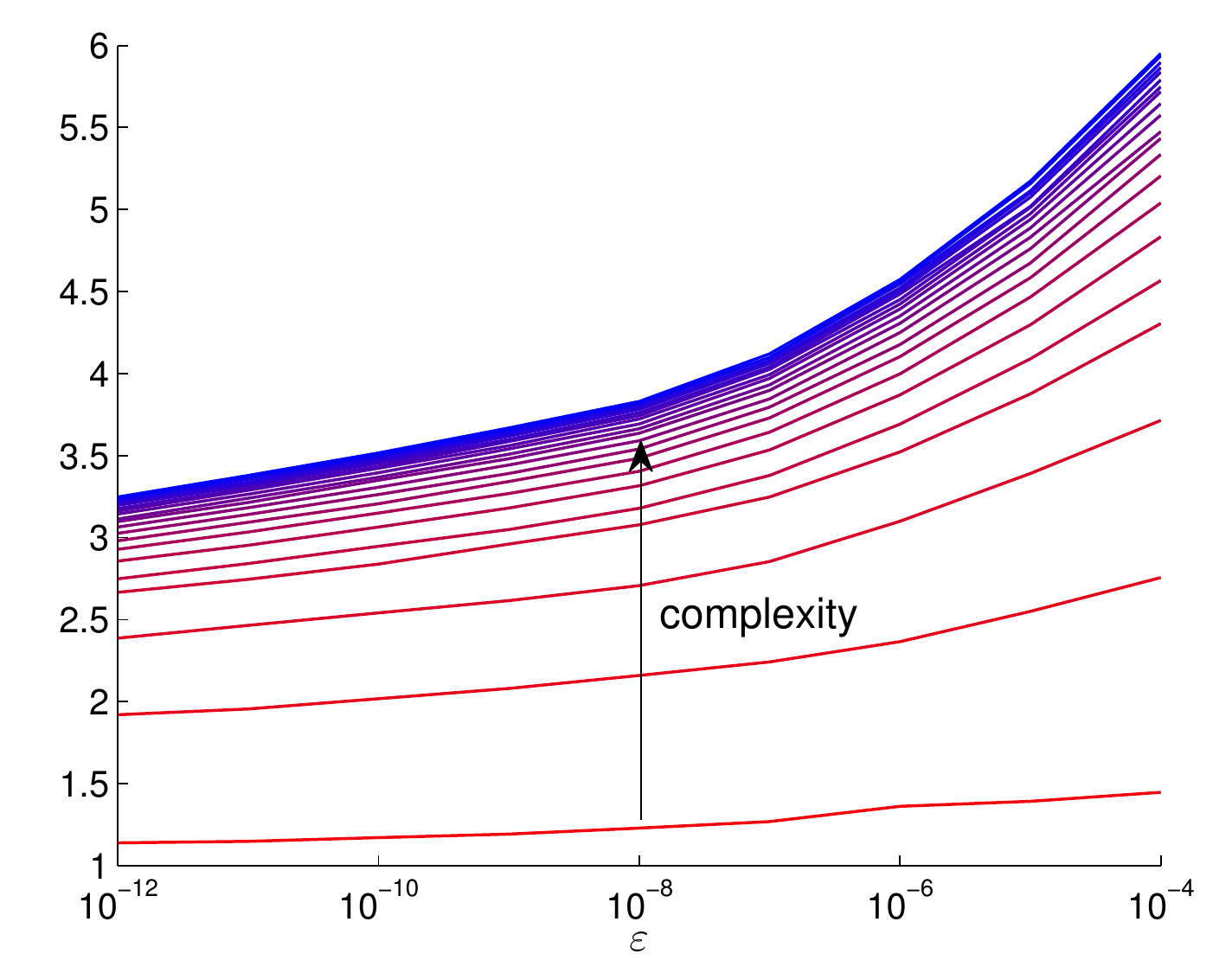}
        \end{minipage}
        \begin{minipage}[c]{0.45\textwidth}
            \centering
            \subfiguretitle{Deviation} \vspace*{0.35em}
            \includegraphics[width=\textwidth]{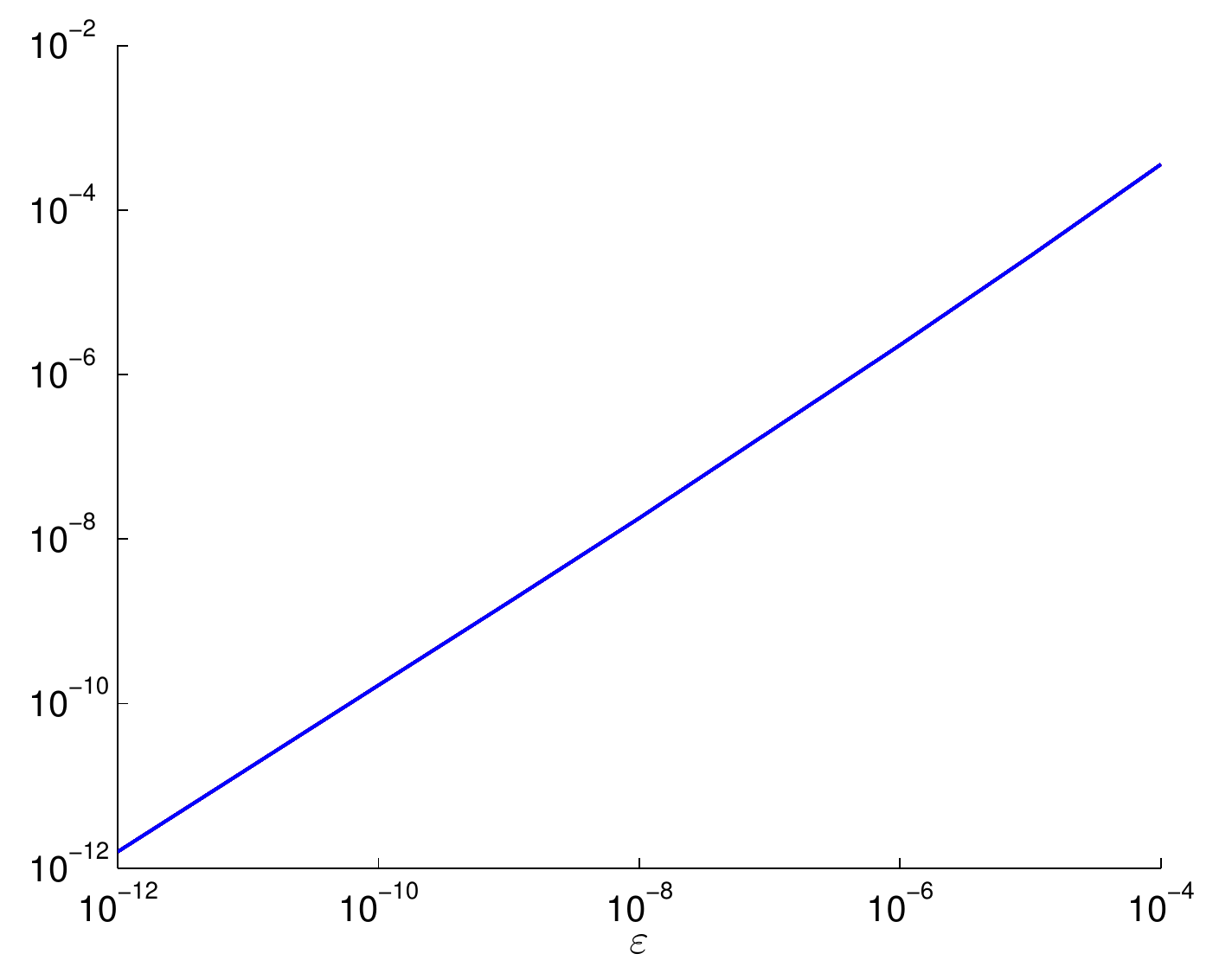}
        \end{minipage}
    \end{center}
    \caption{Speedup and deviation of sfRK as a function of $ \varepsilon $.}
    \label{fig:InverterchainEpsilon_RK}
\end{figure}

Note that the deviation does not depend on the complexity since only artificial terms were introduced to model different complexities of the transistor model.
\end{example}

\subsection{Implicit Runge--Kutta methods}

The stages of implicit Runge--Kutta methods
 cannot be evaluated successively. At each time point, a system of nonlinear equations has to be solved. To solve these systems with the Newton--Raphson method, the Jacobian $ \pd{f_I}{x_I} $ has to be computed. For the transient analysis of integrated circuits, this can be accomplished efficiently using so-called element stamps~\cite{GFtM05}. Every time the right-hand side $ f_I $ is evaluated, the Jacobian $ \pd{f_I}{x_I}$---if needed---is generated simultaneously.

However, only the nonlinear equations that correspond to active regions will be solved assuming that the influence of and on the latent regions is negligibly small. Furthermore, it is then only necessary to compute and factorize the fraction of the Jacobian which represents the active part. That is, we can exploit the latency also on the level of the nonlinear and linear systems of equations. In our implementation, a variable is not updated if it is at least latent of order one, the influence of longer paths is neglected again.

In the following, we will consider in particular the trapezoidal rule, which is frequently used for the simulation of integrated circuits. Since the second version of \textsc{Spice} most circuit simulators apply either the trapezoidal rule or BDF schemes to solve the circuit equations~\cite{GFtM05}. We will denote the trapezoidal rule abbreviatory as TR and the signal-flow based trapezoidal rule as sfTR.

The increment function of the trapezoidal rule tailored to time-driven ordinary differential equations can be written as
\begin{equation}
    \Phi(t^m, x^m, h) = \frac{1}{2}\left(f_I(x_E^m, x_I^m) + f_I(x_E^{m+1}, x_I^{m+1})\right).
\end{equation}
That is, at each time step a system of nonlinear equations
\begin{equation}
    F(z) \coloneqq z - x_I^m - \frac{h}{2}\left(f_I(x_E^m, x_I^m) + f_I(x_E^{m+1}, z)\right) = 0
\end{equation}
has to be solved. Using the Newton--Raphson method, this leads to the iteration
\begin{equation}
    z_{k+1} = z_k + \Delta z_k,
\end{equation}
where $ \Delta z_k $ is the solution of the linear system of equations
\begin{equation}
    \left(I - \frac{h}{2} \pd{f_I}{x_I}(x_E^{m+1}, z_k)\right) \Delta z_{k}
        = -z_k + x_I^m + \frac{h}{2}\left(f_I(x_E^m, x_I^m) + f_I(x_E^{m+1}, z_k)\right).
\end{equation}
As a starting point for the iteration, we use $ z_0 = x_I^m $.

\begin{example} \label{ex:Inverterchain_sfTR}
To facilitate comparisons of the explicit Runge--Kutta method and the implicit trapezoidal rule, we repeat the simulation of the inverter chain of length $ N = 100 $ with the settings described in Example~\ref{ex:Inverterchain_sfRK}. Figure~\ref{fig:Inverterchain_sfTR} shows the runtimes of the simulation with  both the standard trapezoidal rule and the signal-flow based trapezoidal rule for varying model complexities and input functions again. We use the Newton--Raphson method to solve the nonlinear systems and the LU factorization to solve the resulting linear systems of equations. For the signal-flow based simulation, only the active and semi-latent parts of the nonlinear and linear systems of equations are generated and solved. Here, the influence of the model complexity is negligible since the runtime of the LU factorizations is dominating. Table~\ref{tab:Inverterchain_sfTR} contains the number of required transistor model evaluations. The influence of $ \varepsilon $ on the speedup of sfTR and the average deviation per step for a fixed delay $ \Delta T = 10 $ are shown in Figure~\ref{fig:InverterchainEpsilon_TR}.

If the delay $ \Delta T $ of the input function is bigger than $ 12 $ or the period is bigger than $ 14 $, respectively, then the trapezoidal rule depends on the latency. This is due to the fact that the signal needs approximately this period of time to pass all inverters. For larger values of $ \Delta T $, there is a small time interval where all vertices are latent and thus the Newton--Raphson method needs less iterations to converge.

\begin{figure}[htbp]
    \begin{center}
        \begin{minipage}[c]{0.45\textwidth}
            \centering
            \subfiguretitle{TR} \vspace*{0.35em}
            \includegraphics[width=\textwidth]{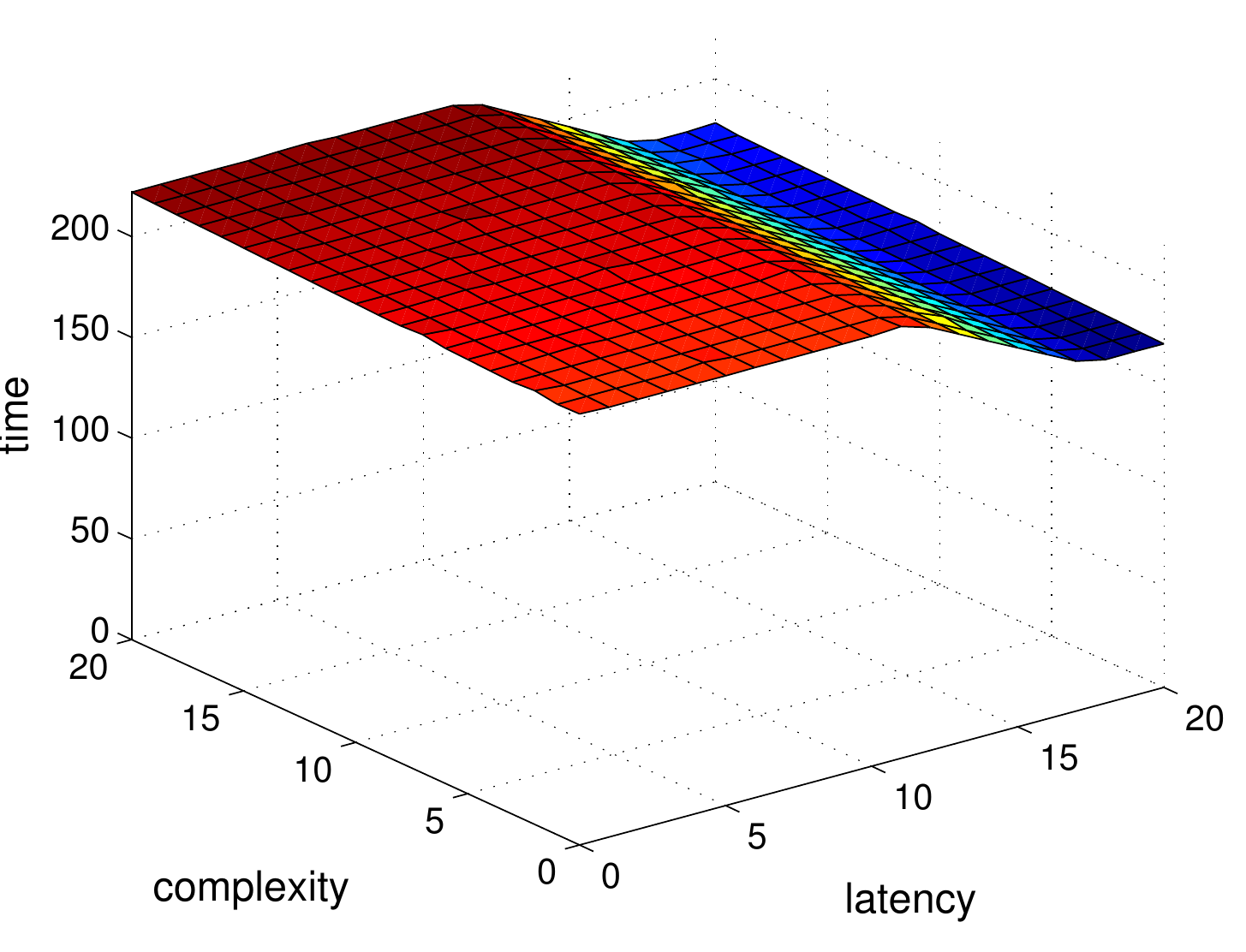}
        \end{minipage}
        \begin{minipage}[c]{0.45\textwidth}
            \centering
            \subfiguretitle{sfTR} \vspace*{0.35em}
            \includegraphics[width=\textwidth]{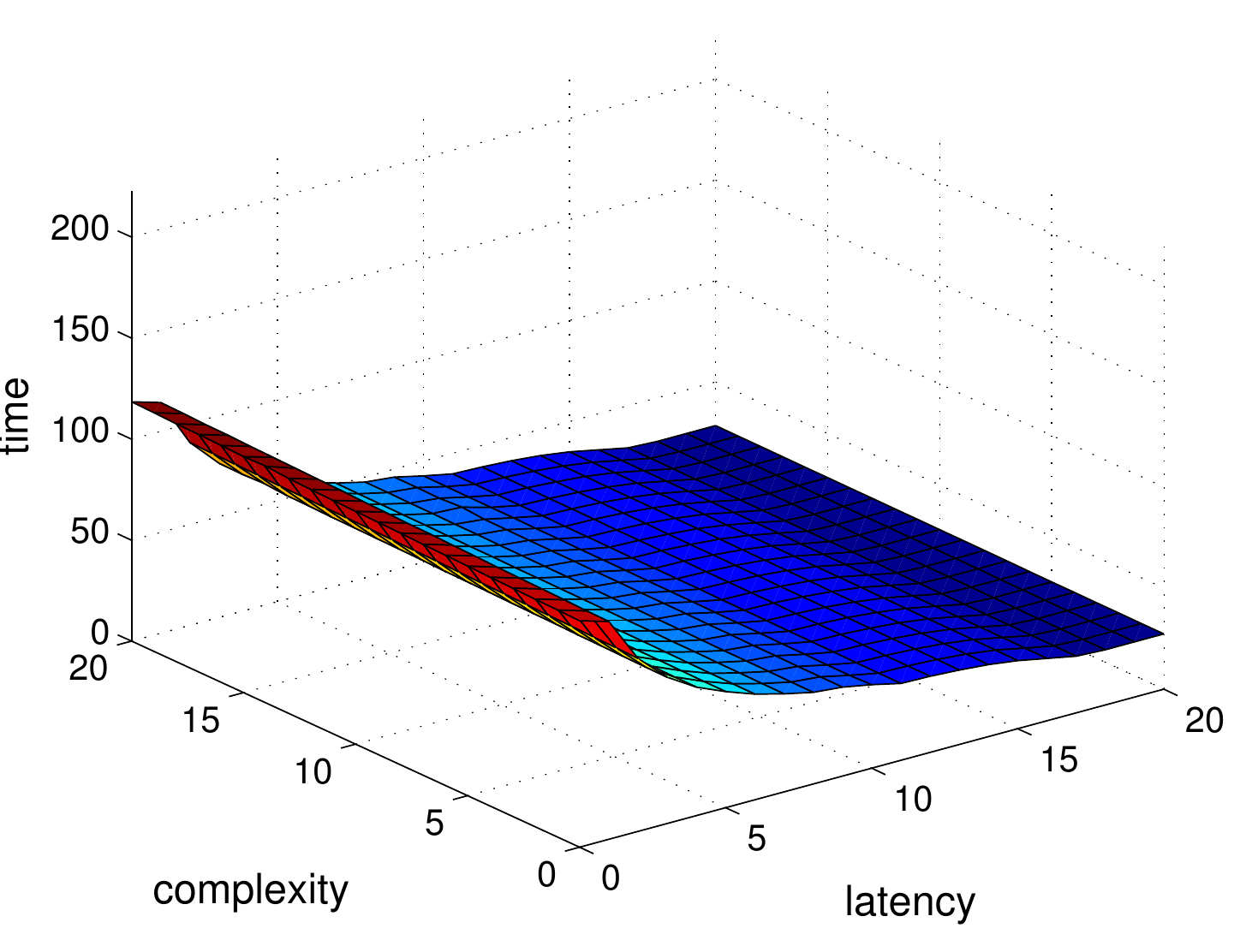}
        \end{minipage}
        \\ \vspace*{3mm}
        \begin{minipage}[c]{0.45\textwidth}
            \centering
            \subfiguretitle{TR vs. sfTR} \vspace*{0.35em}
            \includegraphics[width=\textwidth]{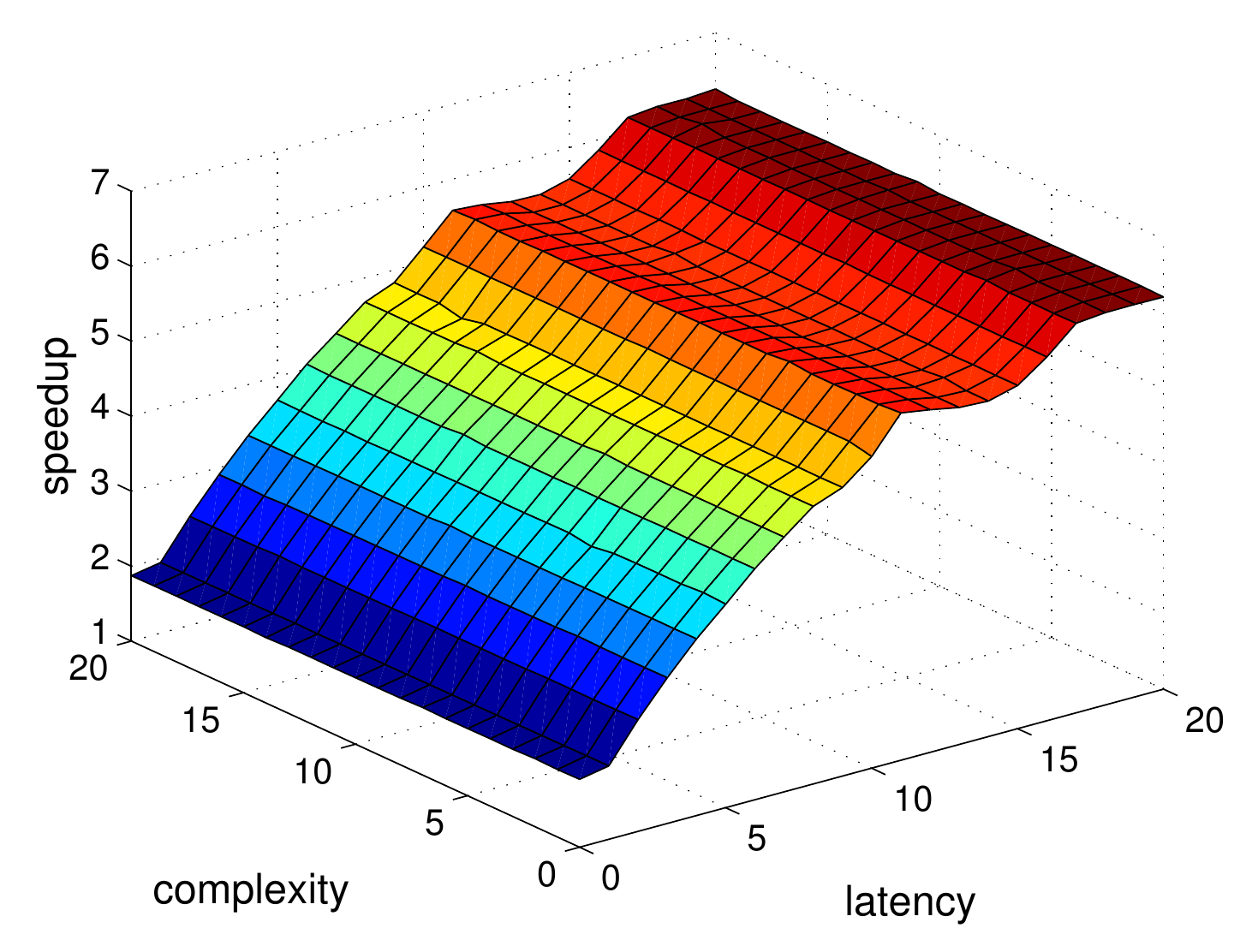}
        \end{minipage}
    \end{center}
    \caption{Influence of the complexity and latency on the runtime of TR and sfTR.}
    \label{fig:Inverterchain_sfTR}
\end{figure}

\begin{table}[htb]
    \caption{Number of transistor model evaluations of TR and sfTR.}
    \newcommand{\mc}[1]{\multicolumn{1}{|c|}{#1}}
    \newcommand{\ts}{\mspace{2mu}}
    \footnotesize
    \centering
    \begin{tabular}{|r*{5}{|r}|}
        \hline
        \mc{$ \Delta T $} & \mc{0} & \mc{5} & \mc{10} & \mc{15} & \mc{20} \\
        \hline
        \hline
                  TR & $ 2\ts353\ts600 $ & $ 2\ts353\ts600 $ & $ 2\ts353\ts600 $ & $ 2\ts075\ts200 $ & $ 1\ts881\ts600 $ \\
                sfTR & $ 1\ts736\ts618 $ & $     784\ts214 $ & $     486\ts788 $ & $     357\ts118 $ & $     307\ts582 $ \\
        \hline
    \end{tabular}
    \label{tab:Inverterchain_sfTR}
\end{table}

\begin{figure}[htb]
    \begin{center}
        \begin{minipage}[c]{0.45\textwidth}
            \centering
            \subfiguretitle{Speedup} \vspace*{0.35em}
            \includegraphics[width=\textwidth]{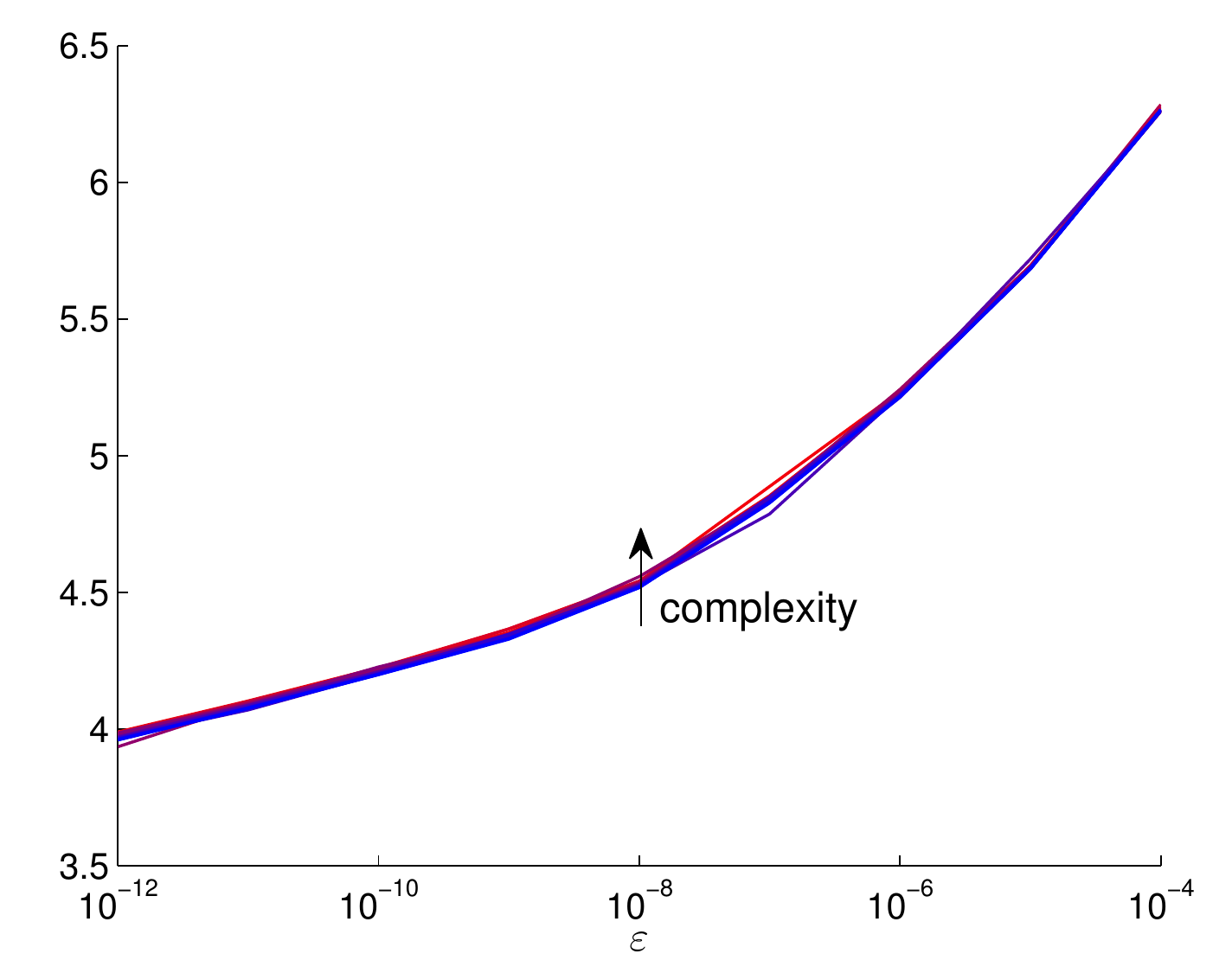}
        \end{minipage}
        \begin{minipage}[c]{0.45\textwidth}
            \centering
            \subfiguretitle{Deviation} \vspace*{0.35em}
            \includegraphics[width=\textwidth]{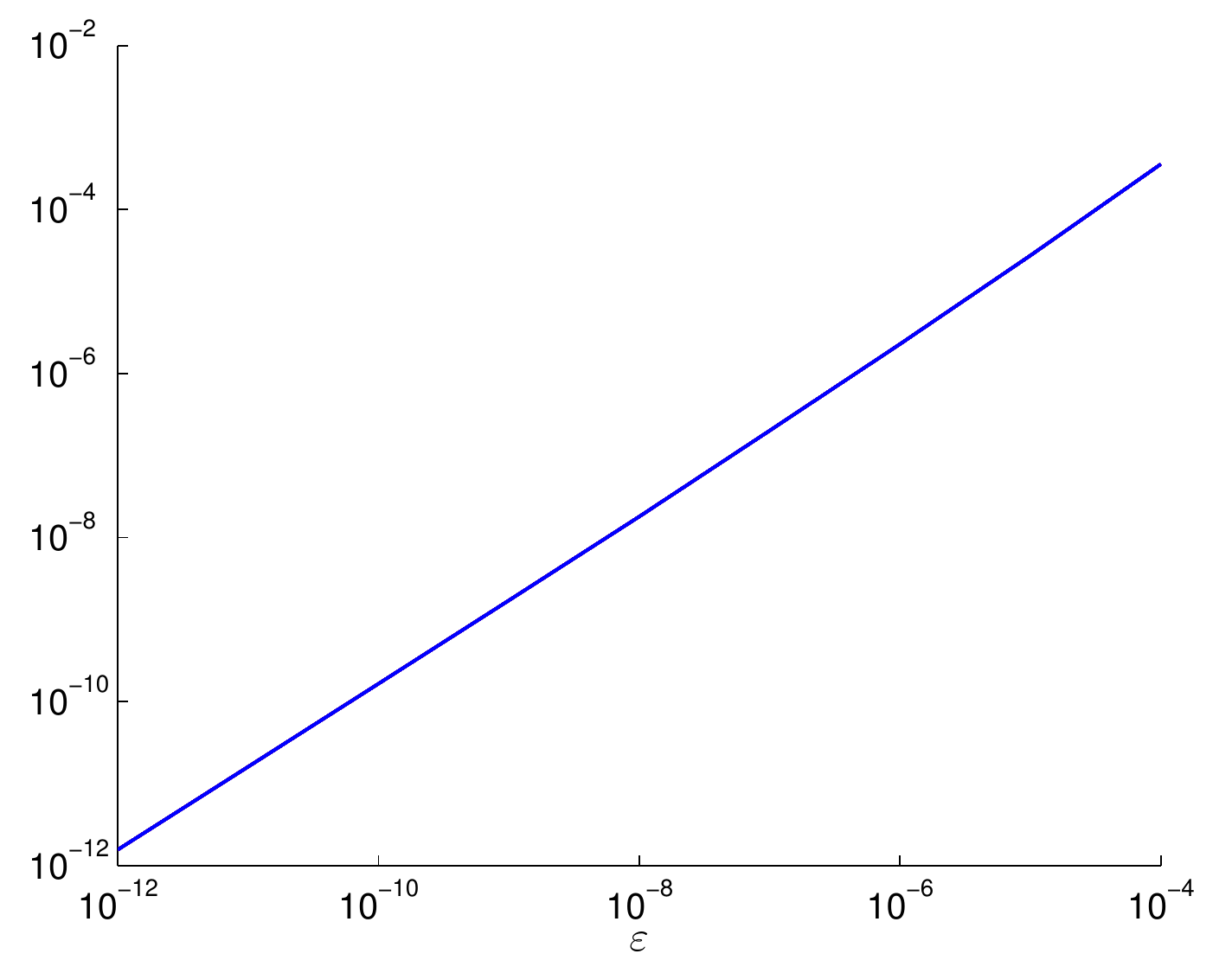}
        \end{minipage}
    \end{center}
    \caption{Speedup and deviation of sfTR as a function of $ \varepsilon $.}
    \label{fig:InverterchainEpsilon_TR}
\end{figure}
\end{example}

\section{Generalization to periodic systems}

In power electronic circuits, diodes and semiconductor switches are constantly changing their status and a steady state condition is by definition reached when the waveforms are periodic with a time period $ T $ which depends on the specific nature of the circuit~\cite{MUR95}. The time scales of these circuits may differ by several orders of magnitude and the simulation requires very small step sizes to cover the dynamics of the fastest subsystems. The maximum simulation time, on the other hand, is usually determined by the slowest subsystems. Thus, a detailed simulation of power electronic circuits is in general very time-consuming. Now, we want to extend the signal-flow based approach to identify and exploit not the latency but the periodicity of subsystems in order to reduce the runtime of the simulation.

\begin{definition}[Semi-periodicity]
Let $ T $ be the fundamental period of the system and $ h = \frac{T}{p} $, $ p \in \mathbb{N} $, the step size.
\begin{enumerate}
\item An external variable $ x_{E, i} $, $ i \in \indices{n_E} $, is said to be \emph{semi-periodic} at $ t^m $ if
\begin{equation}
    f_{E, i}(t^m + c_q h) = f_{E, i}(t^{m-p} + c_q h)
\end{equation}
for all $ q = 1, \dots, s $.
\item An internal variable $ x_{I, i} $, $ i \in \indices{n_I} $, is defined to be \emph{semi-periodic} if
\begin{equation}
    x_{I, i}^m = x_{I, i}^{m-p}.
\end{equation}
\end{enumerate}
\end{definition}

In contrast to the definition of semi-latency, the variables are not compared to the previous time step, but to the corresponding time step of the previous period. Roughly speaking, latency can be regarded as a special case of periodicity for which $ p = 1 $.

\begin{definition}[Periodicity]
A variable $ x_i $, $ i \in \indices{n} $, is called \emph{periodic of order} $ 1 $, if $ x_i $ and all variables of the set $ \pre{x_i} $ are semi-periodic. Additionally, a periodic variable $ x_i $ is defined to be \emph{periodic of order} $ \nu $ if all variables in $ \pre{x_i} $ are at least periodic of order $ \nu-1 $.
\end{definition}

Let $ \varepsilon $ be again a given error tolerance. For numerical computations, the semi-periodicity conditions are replaced by $ \abs{x_{E, i}^m - x_{E, i}^{m-p}} < \varepsilon $ and $ \abs{x_{I, i}^m - x_{I, i}^{m-p}} < \varepsilon $, respectively. Analogously to the latency-based methods, we do not update a variable if it is periodic of order one or higher. To illustrate the different activity states, we use the inverter chain.

\begin{example}
The inverter chain is excited with a piecewise linear function which is periodic with $ T = 4 $ for $ t > 1 $. The input function and the resulting node voltages at intermediate vertices are shown in Figure~\ref{fig:InverterchainSimulationP}.

\begin{figure}[htb]
    \begin{center}
        \begin{minipage}[c]{0.45\textwidth}
            \centering
            \subfiguretitle{Latency}
            \includegraphics[width=\textwidth]{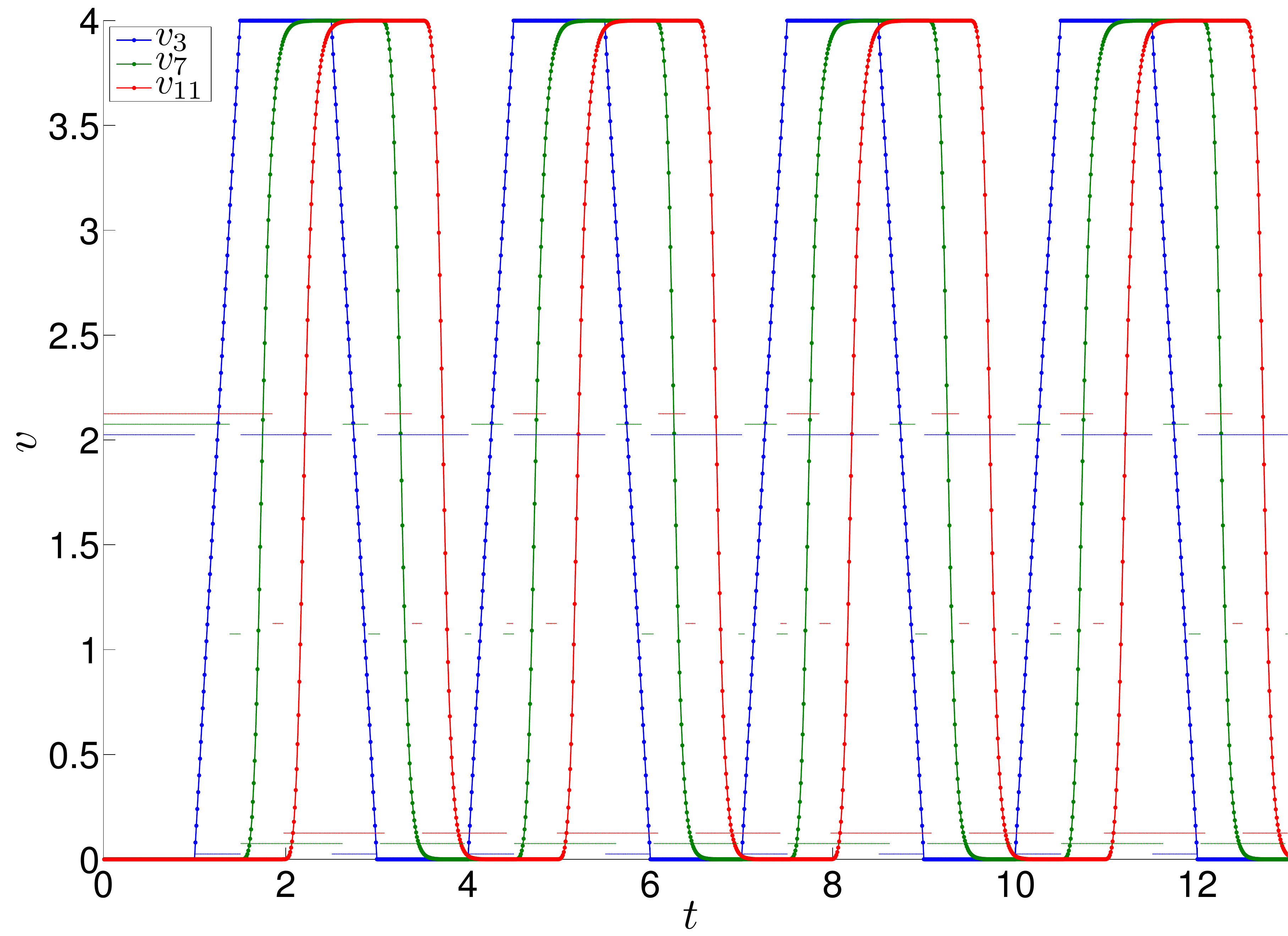}
        \end{minipage}
        \begin{minipage}[c]{0.45\textwidth}
            \centering
            \subfiguretitle{Periodicity}
            \includegraphics[width=\textwidth]{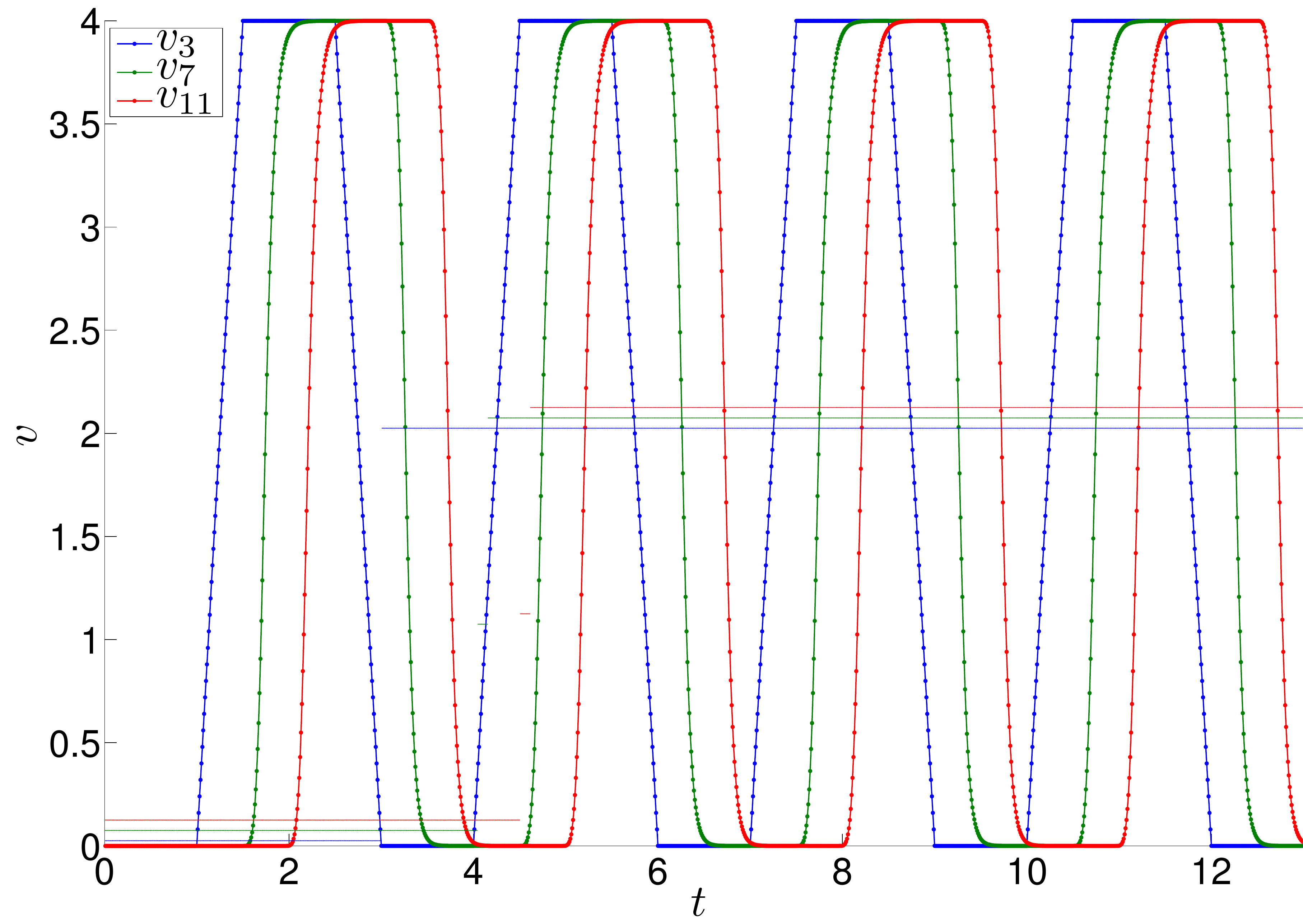}
        \end{minipage}
    \end{center}
   \caption{Comparison of latency and periodicity. The curves show the node voltages $ v_3 $, $ v_7 $,
            and $ v_{11} $, the thin horizontal lines the corresponding states of the variables. Here, $ 0 $
            denotes active, $ 1 $ semi-latent or semi-periodic, and $ 2 $ latent or periodic, respectively.}
    \label{fig:InverterchainSimulationP}
\end{figure}
\end{example}

\begin{definition}[Signal-flow based periodic Runge--Kutta method]
An explicit \emph{signal-flow based periodic Runge--Kutta method} for a time-driven ordinary differential equation is defined by
\begin{equation}
    \begin{split}
        x_E^{m+1} &= x_E^m + \Delta x_E^m, \\
        x_{I, i}^{m+1} &=
        \begin{cases}
            x_{I, i}^{m-p+1},               & \text{if } x_{I, i} \text{ is periodic of order } s, \\
            x_{I, i}^m + \Delta x_{I, i}^m, & \text{otherwise},
        \end{cases}
    \end{split}
\end{equation}
for $ i \in \indices{n_I} $.
\end{definition}

To exploit the periodicity of subsystems and to reduce the number of function evaluations, we store the vectors $ x^{m-p+1}, x^{m-p+2}, \dots, x^m $ in a circular buffer.

\begin{theorem} \label{th:pERK=sfpERK}
The explicit Runge--Kutta methods and the corresponding signal-flow based methods for periodic systems are equivalent.
\end{theorem}
\begin{proof}
\allowdisplaybreaks
The proof is almost identical to the proof of Theorem~\ref{th:ERK=sfERK}. We add again the superscript $ m $ or $ m-p $ to the stages to differentiate between the time points. Let $ x_{I, i} $ be periodic at $ t^m $, i.e.\ $ x_{I, i}^m = x_{I, i}^{m-p} $ and
\begin{align*}
    f_{E, j}(t^m + c_q h) &= f_{E, j}(t^{m-p} + c_q h) \quad \forall x_{E, j} \in \pre{x_{I, i}}, \\
    x_{I, j}^m &= x_{I, j}^{m-p} \quad \forall x_{I, j} \in \pre{x_{I, i}}.
\end{align*}
For $ q = 1 $, this yields
\begin{equation*}
    k_{I, i}^{m, 1} = f_{I, i}(x_E^m, x_I^m) = f_{I, i}(x_E^{m-p}, x_I^{m-p}) = k_{I, i}^{m-p, 1}
\end{equation*}
and hence by induction
\begin{align*}
    k_{I, i}^{m, q} &= f_{I, i}\big(k_E^{m, q}, x_I^m + h \sum_{r=1}^{q-1} a_{qr} k_I^{m, r}\big) \\
                    &= f_{I, i}\big(k_E^{m-p, q}, x_I^{m-p} + h \sum_{r=1}^{q-1} a_{qr} k_I^{m-p, r}\big)
                     = k_{I, i}^{m-p, q}
\end{align*}
for each variable $ x_{I, i} $ which is periodic of order $ q $. Consequently,
\begin{equation*}
    \begin{split}
        x_{I, i}^{m+1} &= x_{I, i}^m + h \, \Phi_i(t^m, x^m, h) \\
                       &= x_{I, i}^m + h \sum_{q=1}^s b_q k_{I, i}^{m, q} \\
                       &= x_{I, i}^{m-p} + h \sum_{q=1}^s b_q k_{I, i}^{m-p, q} \\
                       &= x_{I, i}^{m-p} + h \, \Phi_i(t^{m-p}, x^{m-p}, h) = x_{I, i}^{m-p+1},
    \end{split}
\end{equation*}
for each $ x_{I, i} $ which is periodic of order $ s $.
\end{proof}

Now, let sfpRK denote the signal-flow based standard fourth-order Runge--Kutta method for periodic systems.

\begin{example} \label{ex:Inverterchain_sfpRK}
To compare the signal-flow based method for periodic systems with the standard Runge--Kutta method, we simulate the inverter chain as described in Example~\ref{ex:Inverterchain_sfRK}. The results are shown in Figure~\ref{fig:Inverterchain_sfpRK} and Table~\ref{tab:FunctionEval_sfpRK}. Here, the number of function evaluations rises with increasing $ \Delta T $ since the time interval in which the system is periodic according to our definition decreases.

\begin{figure}[htbp]
    \begin{center}
        \begin{minipage}[c]{0.45\textwidth}
            \centering
            \subfiguretitle{RK} \vspace*{0.35em}
            \includegraphics[width=\textwidth]{InverterchainTime_RK}
        \end{minipage}
        \begin{minipage}[c]{0.45\textwidth}
            \centering
            \subfiguretitle{sfpRK} \vspace*{0.35em}
            \includegraphics[width=\textwidth]{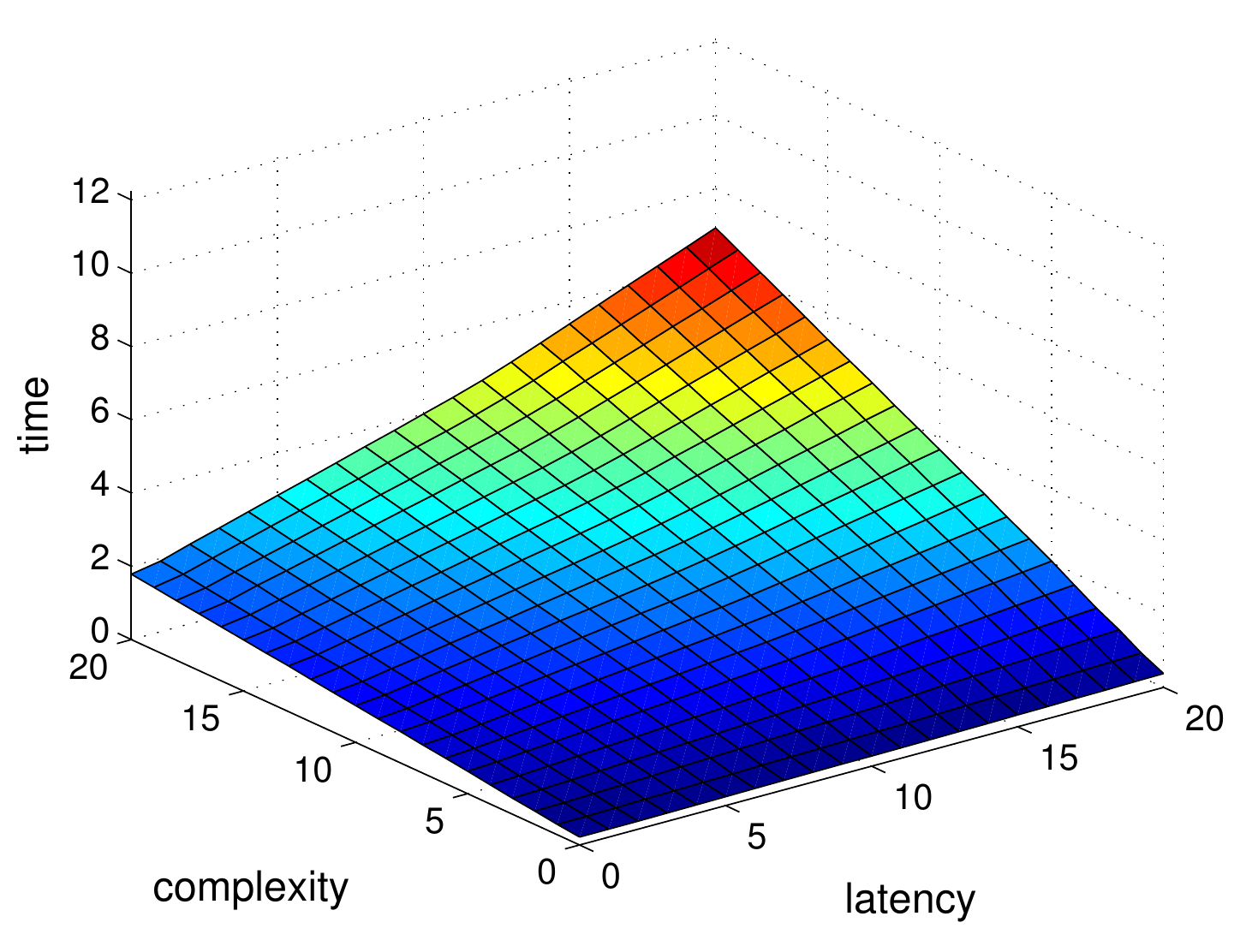}
        \end{minipage}
        \\ \vspace*{3mm}
        \begin{minipage}[c]{0.45\textwidth}
            \centering
            \subfiguretitle{RK vs. sfpRK} \vspace*{0.35em}
            \includegraphics[width=\textwidth]{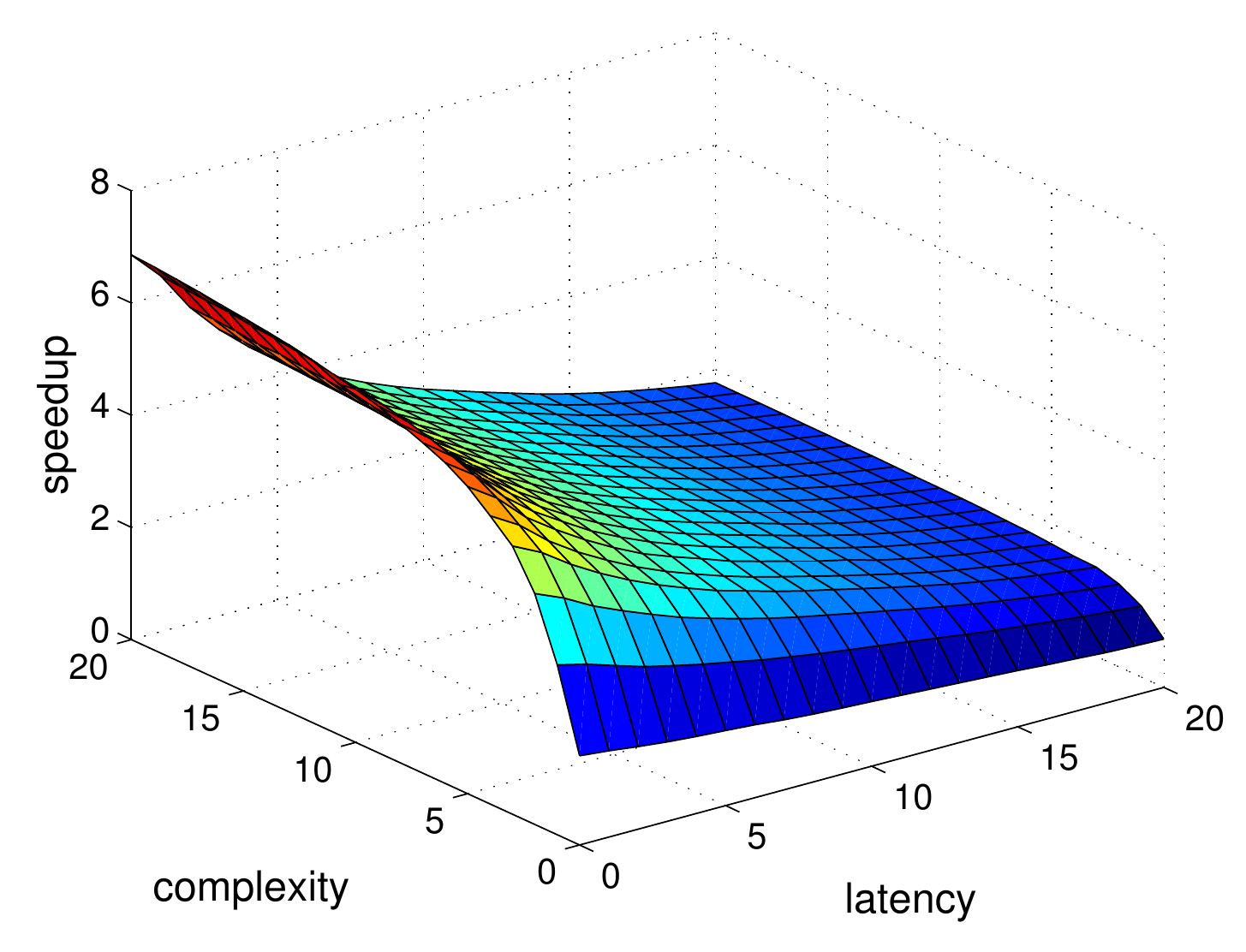}
        \end{minipage}
    \end{center}
    \caption{Influence of the complexity and latency on the runtime of RK and sfpRK.}
    \label{fig:Inverterchain_sfpRK}
\end{figure}

\begin{table}[htb]
    \caption{Number of transistor model evaluations of RK and sfpRK.}
    \newcommand{\mc}[1]{\multicolumn{1}{|c|}{#1}}
    \newcommand{\ts}{\mspace{2mu}}
    \footnotesize
    \centering
    \begin{tabular}{|r*{21}{|r}|}
        \hline
        \mc{$ \Delta T $} & \mc{0} & \mc{5} & \mc{10} & \mc{15} & \mc{20} \\
        \hline
        \hline
              RK  & $ 3\ts200\ts000 $ & $ 3\ts200\ts000 $ & $ 3\ts200\ts000 $ & $ 3\ts200\ts000 $ & $ 3\ts200\ts000 $ \\
            sfpRK & $     422\ts328 $ & $     700\ts936 $ & $     999\ts672 $ & $ 1\ts360\ts800 $ & $ 1\ts760\ts800 $ \\
        \hline
    \end{tabular}
    \label{tab:FunctionEval_sfpRK}
\end{table}
\end{example}

\section{Conclusion}

The efficiency of the signal-flow based Runge--Kutta methods depends strongly on the characteristic properties of the system. The inverter chain example shows that if during the simulation large parts of the system are latent and function evaluations are comparatively time-consuming, then the signal-flow based methods result in a substantially reduced runtime while introducing only a small deviation compared to the corresponding standard Runge--Kutta methods. If, on the other hand, large parts are periodic with a fundamental period $ T $, then the signal-flow based methods for periodic systems can be used to speed up the simulation. The following example summarizes these results.

\begin{example} \label{ex:Inverterchain_sfRK_sfpRK}
Figure~\ref{fig:Inverterchain_sfRK_sfpRK} shows a comparison of the signal-flow based standard Runge--Kutta method and the corresponding method for periodic systems. If $ T $ is small, then the periodicity-oriented Runge--Kutta method is more efficient since the circuit is active most of the time. With increasing $ T $, the latency exploitation becomes more efficient.

\begin{figure}[htbp]
    \begin{center}
        \begin{minipage}[c]{0.45\textwidth}
            \centering
            \subfiguretitle{sfRK vs. sfpRK} \vspace*{0.35em}
            \includegraphics[width=\textwidth]{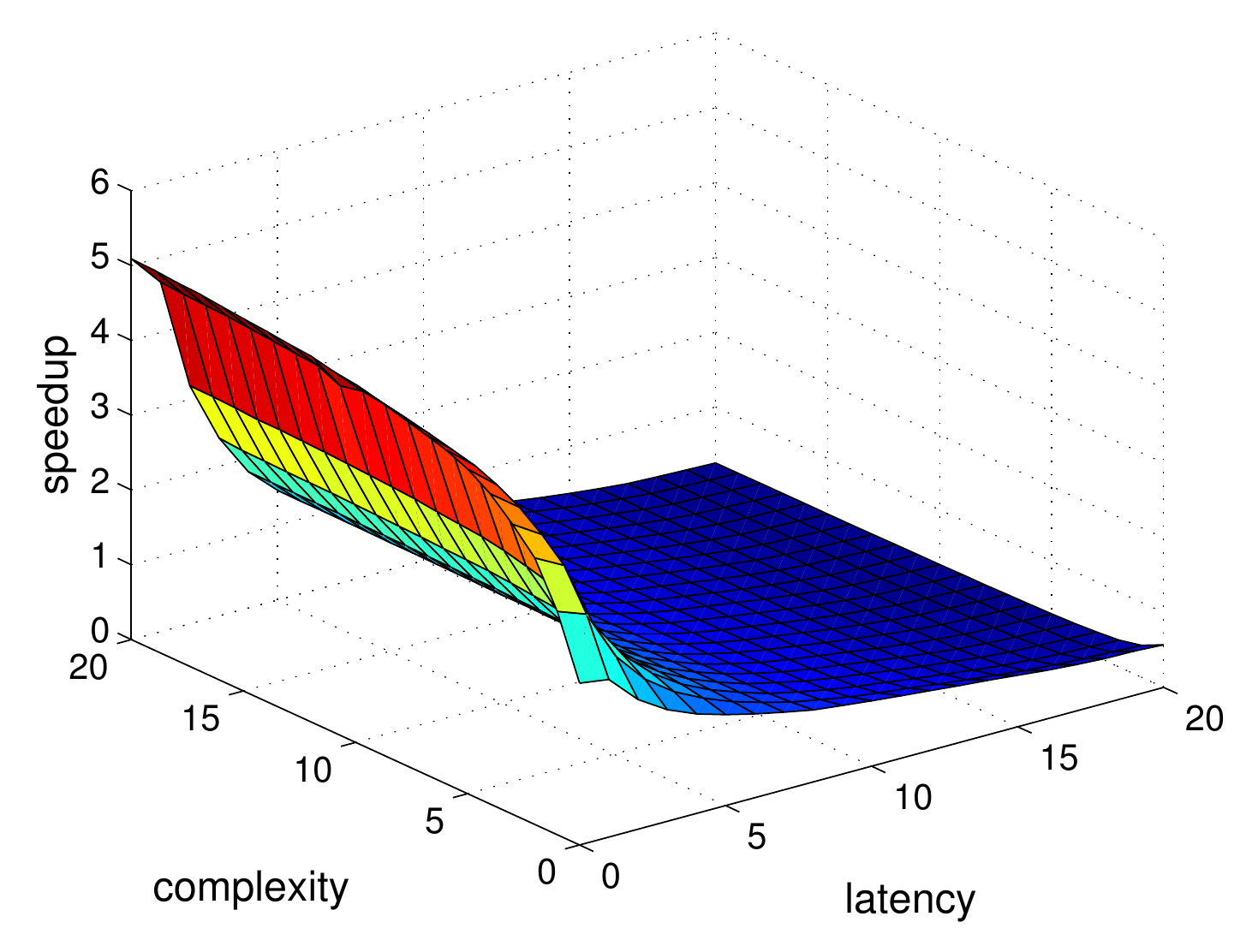}
        \end{minipage}
        \begin{minipage}[c]{0.45\textwidth}
            \centering
            \subfiguretitle{sfpRK vs. sfRK} \vspace*{0.35em}
            \includegraphics[width=\textwidth]{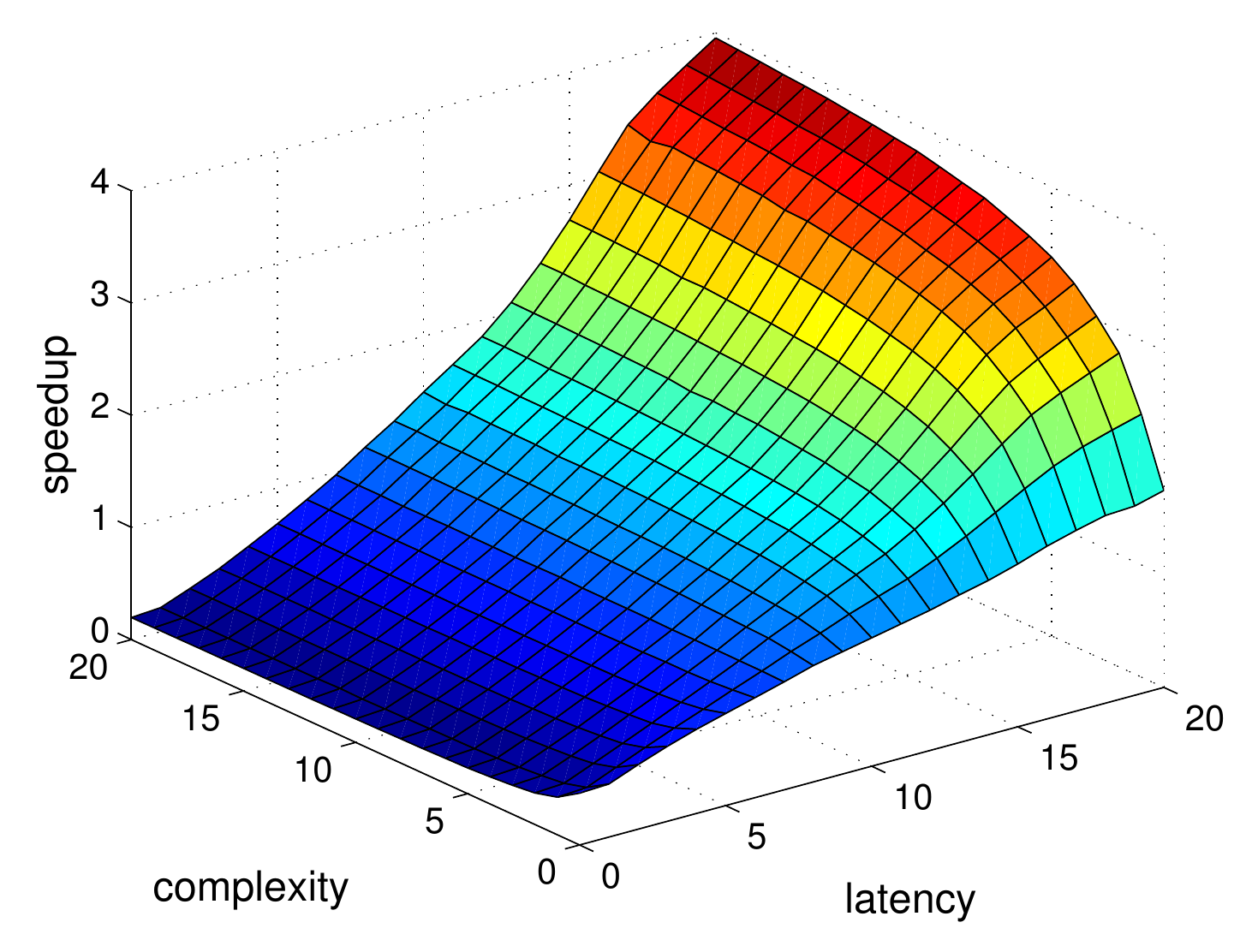}
        \end{minipage}
    \end{center}
    \caption{Comparison of sfRK and sfpRK.}
    \label{fig:Inverterchain_sfRK_sfpRK}
\end{figure}
\end{example}

\section{Further extensions}

To utilize not only the temporal latency, i.e.\ inactivity over a period of time, but also the spatial latency, i.e.\ inactivity during the Newton--Raphson iterations, the proposed techniques might be applicable as well. This could, for example, be used to speed up the DC analysis, exploiting the fact that some parts of the circuit possibly converge rapidly to a solution while other parts converge only very slowly.

%%%%%%%%%%%%%%%%%%%%%%%%%%%%%%%%%%%%%%%%%%%%%%%%%%

\medskip
% The data information below will be filled by AIMS editorial staff
Received xxxx 20xx; revised xxxx 20xx.
\medskip

\end{document}